\documentclass{amsart}

\usepackage{amsmath,amssymb,amsfonts,amsthm}
\usepackage{amsthm}
\usepackage{enumerate}
\numberwithin{equation}{section}
\usepackage{marginnote}
\usepackage{geometry}
\usepackage{pdfrender,xcolor}

\reversemarginpar

\usepackage[english]{babel}
\usepackage[babel, english=american]{csquotes}
\newtheoremstyle{Teorema}
{3pt}
{3pt}
{\slshape}
{}
{\bfseries}
{:}
{\newline}
{}

\newtheorem{theorem}{Theorem}[section]
\newtheorem{definition}[theorem]{Definition}

\newtheorem{lemma}[theorem]{Lemma}

\theoremstyle{definition}

\DeclareMathOperator{\supp}{supp}

\newcommand{\Z}{\mathbb{Z}}

\newcommand{\C}{\mathbb{C}}

\begin{document}


\title[Relative asymptotics of multiple orthogonal polynomials for Nikishin systems]
{Relative asymptotics of multiple orthogonal polynomials for Nikishin systems of two measures}

\author[A. L\'opez Garc\'{i}a]{A. L\'{o}pez Garc\'{i}a}
\address{A. L\'opez Garc\'{i}a\\
Department of Mathematics \\
University of Central Florida \\
4393 Andromeda Loop North \\
Orlando, FL 32816, USA}
\email{abey.lopez-garcia@ucf.edu}

\author[G. L\'{o}pez Lagomasino]{G. L\'{o}pez Lagomasino}\address{ G. L\'{o}pez Lagomasino\\
Departamento de Matem\'aticas\\
Universidad Carlos III de Madrid\\
Avda. Universidad 30, 28911 Legan\'{e}s, Spain}
\email{lago@math.uc3m.es}
\thanks{The second author received finantial support from the Spanish Ministerio de Ciencia, Innovaci\'{o}n y Universidades through grant PGC2018-096504-B-C33}

\begin{abstract} We study the relative asymptotics of two sequences of multiple orthogonal polynomials corresponding to two Nikishin systems of measures on the real line, the second one of which is obtained from the first one perturbing the generating measures with non-negative integrable functions. Each Nikishin system consists of two measures.

\end{abstract}	

\maketitle



\section{Introduction}
\subsection{Background} Let $\mu$ be a finite positive Borel measure supported on $[-1,1]$ that satisfies the Szeg\H{o} condition
\begin{equation} \label{Sze-cond}
\int_{-1}^1 \log \mu'(x) \frac{dx}{\sqrt{1-x^2}} > -\infty,
\end{equation}
where $\mu'$ denotes the Radon-Nikodym derivative of $\mu$ with respect to Lebesgue measure on $[-1,1]$,
and let $p_n(z) = \kappa_n z^n + \cdots, \kappa_n > 0,$ be the orthonormal polynomial of degree $n$ with respect to $\mu$.
Szeg\H{o}'s theorem \cite[Theorem 12.1.2]{szego} states that
\begin{equation} \label{Sz-theorem1} \lim_{n \to \infty} \frac{p_n(z)}{(z + \sqrt{z^2-1})^n} = \frac{1}{\sqrt{2\pi}} S_\mu (z)
\end{equation}
uniformly on compact subsets of $\overline{\C} \setminus [-1,1]$, where $\sqrt{z^2-1}>0$ for $z>1$ and
\[ S_\mu (z) = \exp\left( \frac{\sqrt{z^2 -1}}{2\pi}   \int_{-1}^1  \frac{\log(\sqrt{1 - x^2}\,\mu'(x))}{x - z} \frac{dx}{\sqrt{1 - x^2}}\right)
\]
is the Szeg\H{o} function for the measure $\mu$.

Note that
\[
T_{n}(z)=\frac{1}{2} \left((z + \sqrt{z^2-1})^n +  (z - \sqrt{z^2-1})^n\right)
\]
is the $n$th degree Chebyshev polynomial of the first kind, which is orthogonal with respect to the measure $\frac{dx}{\sqrt{1-x^2}}$ on $[-1,1]$. The corresponding orthonormal polynomials are $t_{n}(x)=\sqrt{\frac{2}{\pi}}\,T_{n}(x)$, $n\geq 1$, $t_{0}(x)=\frac{1}{\sqrt{\pi}}$. On the other hand, $z + \sqrt{z^2-1}$ maps conformally $\overline{\C} \setminus [-1,1]$
onto the complement in $\overline{\C}$ of the closed unit disk. Therefore, \eqref{Sz-theorem1} is equivalent to
\begin{equation} \label{Sz-theorem2} \lim_{n \to \infty} \frac{p_n(z)}{t_n(z)} = S_\mu (z).
\end{equation}
Formula \eqref{Sz-theorem2} indicates that Szeg\H{o}'s theorem can be regarded as an asymptotic relation comparing two sequences of orthogonal polynomials. In general, one can pose the question of studying the relative asymptotics of two sequences of orthogonal polynomials associated with two  measures $\mu_1$ and $\mu_2$. If  $\mu_1$ and $\mu_2$ satisfy \eqref{Sze-cond} the answer is trivial. So, the interest lies in the case when at least one of these measures does not satisfy \eqref{Sze-cond}.

This type of problem was studied for the first time by A. A. Gonchar \cite{Gon1} in the particular case when $d\mu_2 = r\,d\mu_1$, and $r$ is a rational function with complex coefficients such that $r(\infty)=0$ and its poles lie in $\mathbb{C}\setminus[-1,1]$. In the context of positive perturbations $d\mu_2 = g\,d\mu_1$, $g\geq 0$, the problem of relative asymptotics was raised and studied by P. Nevai in \cite{Nevai1}.
Later, it was investigated in great detail (also on the unit circle) in a series of papers by A. M\'{a}t\'{e}, P. Nevai,  and V. Totik, see \cite{MNT1,MNT2,MNT3}, and independently  by E.A. Rakhmanov in \cite{Rak}. Among many results, M\'{a}t\'e-Nevai-Totik proved \cite[Theorem 11]{MNT2}, which we cite next.

Let $\mu$ be a finite positive Borel measure supported on $[-1,1]$ such that $\mu' > 0$ a.e. on that interval. Let $g$ be a non-negative $\mu$-integrable function on $[-1,1]$ such that $|q|g^{\pm 1} \in L^{\infty}(\mu)$ for some polynomial $q\not\equiv 0$. Let $(p_n)_{n\geq 0}$ and $(\widetilde{p}_n)_{\geq 0}$ be the sequences of orthonormal polynomials with respect to $\mu$ and $g\,d\mu$, respectively. Then
\begin{equation} \label{relative} \lim_{n \to \infty}
\frac{\widetilde{p}_n(z)}{p_n(z)}=S_g (z),
\end{equation}
uniformly on compact subsets of $\overline{\C} \setminus [-1,1]$, where
\[
S_g(z) = \exp\left( \frac{\sqrt{z^2 -1}}{2\pi}   \int_{-1}^1  \frac{\log g(x)}{x - z} \frac{d x}{\sqrt{1 - x^2}}\right),
\]
compare with \eqref{Sz-theorem2}. At present, it remains an open problem to find weaker conditions that guarantee the existence of the limit \eqref{relative}.

The M\'{a}t\'{e}-Nevai-Totik theorem was extended in \cite{Gui} to sequences of polynomials orthogonal with respect to varying measures (depending on the parameter $n$ indicating the degree of the polynomials) supported on a bounded interval of the real line, and used to prove the relative asymptotics of orthogonal polynomials for measures supported on unbounded intervals. A slightly more general theorem on the relative asymptotics of polynomials orthogonal with respect to varying measures, of which Lemma \ref{relescalar} below is a corollary, was proved in \cite{berndolyo} (see also \cite{BerLop2,BernardoGuillermo2}) and will be used here to obtain the relative asymptotics of polynomials that satisfy orthogonality relations with respect to a system of two measures. This extension of Szeg\H{o}'s theory is the purpose of this paper.

\subsection{Nikishin systems of two measures and multiple orthogonal polynomials}
Let $\sigma_{1}$ and $\sigma_{2}$ be a pair of finite positive Borel measures on the real line, each with compact and infinite support. Let $\Delta_{k}=\mathrm{Co}(\supp(\sigma_{k}))$ denote the convex hull of the support of $\sigma_{k}$, and assume that $\Delta_{1}\cap \Delta_{2}=\emptyset$. The measures $\sigma_{1}, \sigma_{2}$ will be used to generate our reference (unperturbed) Nikishin system. We say that $(s_{1}, s_{2})=\mathcal{N}(\sigma_{1}, \sigma_{2})$ is the Nikishin system generated by $(\sigma_{1}, \sigma_{2})$ if $s_{1}=\sigma_{1}$ and
\[
d s_{2}(x)=\int\frac{d\sigma_{2}(t)}{x-t}\,d\sigma_{1}(x).
\]
Note that $s_{1}$ and $s_{2}$ are both supported on $\Delta_{1}$.

In this paper, $\Z_{+}$ indicates the set of all non-negative integers. Consider the Nikishin system $(s_{1},s_{2}) = {\mathcal{N}}(\sigma_1,\sigma_2)$ and a multi-index $\mathbf{n}=(n_{1},n_{2})\in\Z_+^2$. It is well known (see \cite{FidLop}) that there exists a unique monic polynomial $Q_{\bf n}$ with $\deg Q_{\bf n}= |{\bf n}| := n_1 + n_2,$ such that
\begin{equation}\label{eqmultort}
\int x^{\nu} Q_{\bf n}(x)\, d s_{k}(x) = 0, \qquad \nu = 0,\ldots,n_k -1, \qquad k=1, 2.
\end{equation}
Now, let $\rho_k$, $k=1, 2$, be non-negative measurable functions defined on $\supp(\sigma_{k})$ such that $\rho_{k}\in L^{1}(\sigma_{k})$, and consider the perturbed Nikishin system $(\widetilde{s}_{1},\widetilde{s}_{2})={\mathcal{N}}(\widetilde{\sigma}_1,\widetilde{\sigma}_2)$, where $d\widetilde{\sigma}_k = \rho_k\, d\sigma_k$. Let $\widetilde{Q}_{\bf n}$ be the monic multiple orthogonal polynomial with respect to
${\mathcal{N}}(\widetilde{\sigma}_1,\widetilde{\sigma}_2)$ and $\mathbf{n}$. We are interested in finding appropriate conditions on the measures $\sigma_k$ and the functions $\rho_k$ which guarantee the existence of the limit
\[
\lim_{{\bf n} \in \Lambda} \frac{\widetilde{Q}_{\bf n}(z)}{{Q}_{\bf n}(z)}
\]
uniformly on compact subsets of $\overline{\C} \setminus \Delta_1$, where $\Lambda \subset \Z_+^2$ is an infinite sequence of distinct multi-indices. In other words, we want to find the relative asymptotics of the corresponding sequences of multiple orthogonal polynomials when we perturb a given Nikishin system of two measures.

This problem was studied in \cite{LL} when the perturbation is given by rational functions (see Theorems 1.1, 5.1, and Corollaries 4.1, 5.1 in the referred paper). Here, we will show that with assumptions analogous to those in the M\'{a}t\'e-Nevai-Totik theorem stated above, there is convergence and the limit is given in terms of appropriate Szeg\H{o} functions.

We will restrict our attention to the class of multi-indices given by
\[ \mathbb{Z}_+^2(\circledast) = \{{\bf n} = (n_1,n_2) \in \mathbb{Z}_+^2: n_2 \leq n_1
+1\}\,.
\]

\begin{definition} \label{secondtype}
The functions of the second kind associated with $\mathcal{N}(\sigma_1,\sigma_2)$ and $\mathbf{n}$ are
\[
\Psi_{{\bf n},0}(z)=Q_{{\bf n}}(z),\qquad\Psi_{{\bf n},k}(z)=\int
\frac{\Psi_{{\bf n},k-1}(x)}{z-x} d\sigma_k(x),\qquad k=1, 2.
\]
\end{definition}

If ${\bf n} \in \mathbb{Z}_+^2(\circledast)$, it is well known (see \cite[Proposition 3]{GRS}) that $\Psi_{{\bf n},k-1}$, $k=1, 2$, has exactly $N_{\mathbf{n},k}:=\sum_{j=k}^{2} n_{j}$ zeros in ${\C}\setminus\Delta_{k-1}$ $(\Delta_0 = \emptyset)$, they are all simple and lie on $\Delta_k$.
Let $Q_{{\bf n},k}$, $k=1, 2,$ be the monic polynomial of degree $N_{\mathbf{n},k}$ whose roots are the zeros of $\Psi_{{\bf n},k-1}$ on $\Delta_k$, and for notational convenience we set $Q_{\textbf{n},0}\equiv Q_{\textbf{n},3}\equiv 1$. Note that $Q_{\textbf{n},1}=Q_{\textbf{n}}$.  Analogously, let $\widetilde{Q}_{\textbf{n},k}$, $\widetilde{\Psi}_{\textbf{n},k}$, $0\leq k\leq 2$, denote the corresponding polynomials and functions of the second kind with respect to $\mathcal{N}(\widetilde{\sigma}_1, \widetilde{\sigma}_2)$.

Before we state our main result, we need to introduce some key definitions and constructions inspired by the works  \cite{sasha,luisyo,luisyo2}. Let $\Delta=(\Delta_{1},\Delta_{2})$, and set
\begin{align*}
\mathbf{C}_{\Delta} & :=\{(f_{1},f_{2}): f_{1}, f_{2}\,\,\textrm{are defined and continuous on}\,\,\Delta_{2}, \Delta_{1},\,\textrm{respectively}\},\\
\mathbf{C}_{\Delta}^{+} & :=\{(f_{1},f_{2})\in \mathbf{C}_{\Delta}: f_{1}>0\,\,\textrm{on}\,\,\Delta_{2}\,\,\textrm{and}\,\,f_{2}>0\,\,\textrm{on}\,\,\Delta_{1}\}.
\end{align*}
For a vector $\mathbf{f}=(f_{1},f_{2})\in\mathbf{C}_{\Delta}$, set
\[
\|\mathbf{f}\|_{\Delta}:=\max\{\|f_{1}\|_{\Delta_{2}}, \|f_{2}\|_{\Delta_{1}}\}
\]
where $\|\cdot\|_{\Delta_{j}}$ denotes the sup norm on $\Delta_{j}$. On $\mathbf{C}_{\Delta}^{+}$ we define the distance
\[
\mathrm{d}(\mathbf{f},\mathbf{g}):=\max\{\|\log(f_{1}/g_{1})\|_{\Delta_{2}}, \|\log(f_{2}/g_{2})\|_{\Delta_{1}}\}.
\]
It is easy to check that $(\mathbf{C}_{\Delta}^{+},\mathrm{d})$ is a complete metric space. Furthermore, if $(\mathbf{g}_{n})_{n\geq 1}$, $\mathbf{g}$ are vector functions in $\mathbf{C}_{\Delta}^{+}$, then
\begin{equation}\label{equivconv}
\lim_{n\rightarrow\infty}\|\mathbf{g}_{n}-\mathbf{g}\|_{\Delta}=0\quad \Leftrightarrow \quad \lim_{n\rightarrow\infty}\mathrm{d}(\mathbf{g}_{n},\mathbf{g})=0.
\end{equation}

If $\rho$ is a non-negative measurable function on the interval $[a,b]$ such that $\log\rho$ is integrable with respect to the measure
$dx/\sqrt{(b-x)(x-a)}$, we write
\begin{equation}\label{def:Szegofunc}
\mathsf{G}_{[a,b]}(\rho;z) := \exp\left(\frac{\sqrt{(z-b)(z-a)}}{2\pi}\int_{a}^{b}\frac{\log\rho(x) }{x-z}\frac{d x}{\sqrt{(b-x)(x-a)}}\right),\qquad z \in \overline{\mathbb{\C}} \setminus [a,b],
\end{equation}
where the square root outside the integral is positive for $z>b$. It is well known (see e.g. \cite{luisyo}) that $ \mathsf{G}(\rho;z)=\mathsf{G}_{[a,b]}(\rho;z)$ is analytic in $\overline{\mathbb{C}}\setminus[a,b]$ and satisfies
\[
\begin{cases}
\mathsf{G}(\rho; z)\neq 0\,\,\,\,\textrm{for}\,\,z\in \overline{\C}\setminus [a,b],\\
\mathsf{G}(\rho;\infty)>0,\\
\lim_{y \to 0} |\mathsf{G}(\rho;x+iy)|^2=\rho^{-1}(x)\,\,\,\, \textrm{for a.e.}\,\,x \in [a,b].
\end{cases}
\]

In order to determine the limiting functions in our main result, the following operator plays a central role.

\begin{definition} \label{def:T}
Let $\rho_{k}$, $k=1, 2,$ be non-negative measurable functions on $\Delta_{k}=\mathrm{Co}(\supp(\sigma_{k}))=[a_{k},b_{k}]$ such that $\log \rho_{k}$ is integrable with respect to the measure $dx/\sqrt{(b_{k}-x)(x-a_{k})}$. Define
\[ T: {\mathbf C}^+_{\Delta} \rightarrow {\mathbf C}^+_{\Delta}
\]
as follows. Given $\mathbf{f}=(f_{1},f_{2})\in {\mathbf C}^+_{\Delta}$, let $T(\mathbf{f})=\mathbf{f}^{*}$, where $\mathbf{f}^{*}=(f_{1}^{*},f_{2}^{*})$ and
\begin{align*}
f_{1}^{*}(x) & :=\mathsf{G}_{\Delta_{1}}(\rho_{1}/f_{2};x),\qquad x\in\Delta_{2},\\
f_{2}^{*}(x) & :=\mathsf{G}_{\Delta_{2}}(\rho_{2}/f_{1};x),\qquad x\in\Delta_{1},
\end{align*}
see \eqref{def:Szegofunc}.
\end{definition}

It is not difficult to show, using the maximum principle for harmonic functions, that the operator $T$ is a contraction on $({\mathbf C}^+_{\Delta},\mathrm{d})$ (for a proof in a more general setting, see \cite[Theorem 1.6]{luisyo2}). In fact, for all $\mathbf{f}, \mathbf{g}\in\mathbf{C}_{\Delta}^{+}$,
\begin{equation}\label{contractionprop}
\mathrm{d}(T(\mathbf{f}),T(\mathbf{g})) \leq \frac{1}{2}\, \mathrm{d}(\mathbf{f},\mathbf{g}).
\end{equation}
Therefore, by the Banach fixed-point theorem, $T$ has a unique fixed-point in ${\mathbf C}^+_{\Delta}$, which we denote by $\Phi=(\Phi_{1},\Phi_{2})\in{\mathbf C}^+_{\Delta}$. Since $\Phi_{k}$ is a Szeg\H{o} function, it has an analytic continuation to $\overline{\mathbb{C}}\setminus\Delta_{k}$, which we also denote by $\Phi_{k}$. Our main result is the following theorem.

\begin{theorem}\label{theomain}
Suppose that $\sigma_{k}'>0$ almost everywhere on $\Delta_{k}$, $k=1,2$. Let $\rho_{k}$ be non-negative measurable functions on $\Delta_{k}$ such that $\rho_{k}\in L^{1}(\sigma_{k})$, and assume that there exists a polynomial $q\not\equiv 0$ such that $|q|\rho_k^{\pm 1} \in L^{\infty}(\sigma_k)$ for $k=1,2$.  Let $\Lambda \subset \mathbb{Z}_+^2(\circledast)$ be an infinite sequence of distinct multi-indices such that $\sup_{{\bf n} \in \Lambda} (n_1 - n_2)<\infty$. Let $(\Phi_1,\Phi_2)$ be the unique fixed point of the operator $T$ given in Definition \ref{def:T}, associated with the weights $(\rho_{1},\rho_{2})$. We have
\begin{align}
\lim_{{\bf n} \in \Lambda} \frac{\widetilde{Q}_{{\bf n},k}(z)}{{Q}_{{\bf n},k}(z)} & = \frac{\Phi_{k}(z)}{\Phi_{k}(\infty)}, \qquad z\in\overline{\mathbb{C}}\setminus\Delta_{k},\quad k=1, 2,\label{limfund1}\\
\lim_{{\bf n} \in \Lambda} \frac{\widetilde{\Psi}_{{\bf n},1}(z)}{{\Psi}_{{\bf n},1}(z)} & = \frac{1}{\Phi_{1}(\infty)}\frac{\Phi_{2}(z)}{\Phi_1(z)},\qquad z\in\overline{\mathbb{C}}\setminus(\Delta_{1}\cup\Delta_{2}), \label{limfund2}\\
\lim_{{\bf n} \in \Lambda} \frac{\widetilde{\Psi}_{{\bf n},2}(z)}{{\Psi}_{{\bf n},2}(z)} & = \frac{1}{\Phi_{1}(\infty)\Phi_{2}(z)},\qquad z\in\overline{\mathbb{C}}\setminus\Delta_{2}, \label{limfund4}
\end{align}
uniformly on compact subsets of the indicated regions.
\end{theorem}

Observe that the stated conditions guarantee that the function $\log\rho_{k}$ is integrable with respect to $dx/\sqrt{(b_{k}-x)(x-a_{k})}$.

A natural question is whether or not a result similar to Theorem \ref{theomain} is true for Nikishin systems generated by
an arbitrary number of $m \geq 2$ measures. This seems to be true, but there is a technical problem still to be solved which we explain at the end of the paper.

The paper is organized as follows. In Section \ref{sec:aux} we gather all the auxiliary results that are necessary for the proof of Theorem~\ref{theomain}. The proof of that theorem is given in Section \ref{section:proof}.

\section{Auxiliary results}\label{sec:aux}

Let
\begin{equation}\label{Hnk}
\mathcal{H}_{{\bf n},k}(z) := \frac{Q_{{\bf n},k-1}(z)\,\Psi_{{\bf n},k-1}(z)}{Q_{{\bf n},k}(z)},\qquad 1\leq k\leq 3.
\end{equation}
It is well known (see \cite{GRS}, or formulas (48) and (50) in \cite{BerLop}) and easy to verify that
\begin{equation}\label{ortogonalidades}
\int x^{\nu} Q_{{\bf n},k}(x) \frac{\mathcal{H}_{{\bf n},k}(x)\, d \sigma_k(x)}{Q_{{\bf n},k-1}(x)\,Q_{{\bf n},k+1}(x)} = 0,\qquad \nu = 0,\ldots,N_{\mathbf{n},k}-1, \qquad k=1, 2,
\end{equation}
and
\begin{equation}\label{entreHnks}
\mathcal{H}_{{\bf n},k+1}(z)  = \int   \frac{Q^2_{{\bf n},k}(x)}{z-x} \frac{\mathcal{H}_{{\bf n},k}(x)\,d\sigma_k(x)}{Q_{{\bf n},k-1}(x)\,Q_{{\bf n},k+1}(x)}, \qquad k=1, 2,\qquad \mathcal{H}_{\textbf{n},1}\equiv 1,
\end{equation}
$Q_{{\bf n},0}\equiv Q_{\textbf{n},3}\equiv 1$.

\begin{theorem}\label{unicidad} Given $\mathcal{N}(\sigma_1,\sigma_2)$ and ${\bf n}=(n_{1},n_{2})\in \mathbb{Z}_+^2(\circledast)$, there exists a unique vector of monic polynomials ${\mathbf{L}}_{\bf n} = (L_{{\bf n},1},L_{{\bf n},2})$ satisfying the following two properties: 1) for each $k=1, 2$, $\deg L_{{\bf n},k} = N_{\mathbf{n},k}$ and the zeros of $L_{\mathbf{n},k}$ are all contained in $\Delta_{k}$; 2) the orthogonality conditions
\begin{equation}\label{ortogonalidades:new}
\int x^{\nu} L_{{\bf n},k}(x) \frac{H_{{\bf n},k}(x)\, d \sigma_k(x)}{L_{{\bf n},k-1}(x)\,L_{{\bf n},k+1}(x)} = 0,\qquad \nu = 0,\ldots,N_{\mathbf{n},k}-1, \qquad k=1, 2,
\end{equation}
are satisfied, where the functions $H_{\mathbf{n},k}$ are defined recursively using
\begin{equation}\label{entreHnks:new}
H_{{\bf n},k+1}(z)  = \int   \frac{L^2_{{\bf n},k}(x)}{z-x} \frac{H_{{\bf n},k}(x)\,d\sigma_k(x)}{L_{{\bf n},k-1}(x)\,L_{{\bf n},k+1}(x)}, \qquad k=1, 2,\qquad H_{\mathbf{n},1}\equiv 1,
\end{equation}
starting from $H_{\mathbf{n},1}\equiv 1$ ($L_{{\bf n},0}\equiv L_{\mathbf{n},3}\equiv 1$).
\end{theorem}
\begin{proof}
The existence of the required vector follows from \eqref{ortogonalidades} and \eqref{entreHnks}. Let ${\mathbf{L}}_{\bf n} = (L_{{\bf n},1},L_{{\bf n},2})$ be an arbitrary vector satisfying properties $1)$ and $2)$. We show first that $L_{\mathbf{n},1}$ must be the monic multiple orthogonal polynomial with respect to $\mathcal{N}(\sigma_1,\sigma_2)$ and $\mathbf{n}$, i.e. it satisfies \eqref{eqmultort}.

From \eqref{entreHnks:new} we easily deduce that for each $1\leq k\leq 3$, the function $H_{\mathbf{n},k}$ is analytic in $\mathbb{C}\setminus\Delta_{k-1}$ ($\Delta_{0}=\emptyset$), it never vanishes in this region, and it is real-valued on $\mathbb{R}\setminus\Delta_{k-1}$. Let us define the functions
\begin{equation}\label{defpsik}
\psi_{k}(z):=\frac{L_{\textbf{n},k+1}(z)\,H_{\textbf{n},k+1}(z)}{L_{\textbf{n},k}(z)},\qquad 0\leq k\leq 2.
\end{equation}
We have $\psi_{0}=L_{\textbf{n},1}$, and it follows from our assumptions that $\psi_{k}$ is analytic in $\mathbb{C}\setminus\Delta_{k}$, and
\begin{equation}\label{estinfpsik}
\psi_{k}(z)=O\left(\frac{1}{z^{n_{k}+1}}\right),\quad z\rightarrow\infty,\qquad k=1, 2.
\end{equation}
Using the notation in \eqref{defpsik}, formulas \eqref{ortogonalidades:new} and \eqref{entreHnks:new} can be rewritten as
\begin{equation}\label{newexp1}
\int x^{\nu}\psi_{k-1}(x)\,\frac{d\sigma_{k}(x)}{L_{\textbf{n},k+1}(x)}=0,\qquad 0\leq \nu\leq N_{\textbf{n},k}-1,\qquad k=1, 2,
\end{equation}
and
\begin{equation}\label{newexp2}
\psi_{k}(z)=\frac{L_{\textbf{n},k+1}(z)}{L_{\textbf{n},k}(z)}\int\frac{L_{\textbf{n},k}(x)}{z-x}\,\psi_{k-1}(x)\,\frac{d\sigma_{k}(x)}{L_{\textbf{n},k+1}(x)},\qquad k=1, 2.
\end{equation}

Since $\deg{L_{\textbf{n},k}}=N_{\textbf{n},k}$, from \eqref{newexp1} we obtain
\[
\int\frac{L_{\textbf{n},k}(z)-L_{\textbf{n},k}(x)}{z-x}\,\psi_{k-1}(x)\,\frac{d\sigma_{k}(x)}{L_{\textbf{n},k+1}(x)}=0
\]
which gives
\[
L_{\textbf{n},k}(z)\int\frac{\psi_{k-1}(x)}{z-x}\,\frac{d\sigma_{k}(x)}{L_{\textbf{n},k+1}(x)}=\int\frac{L_{\textbf{n},k}(x)}{z-x}\,\psi_{k-1}(x)\,\frac{d\sigma_{k}(x)}{L_{\textbf{n},k+1}(x)}
\]
so by \eqref{newexp2} we get
\begin{equation}\label{recformpsi}
\int\frac{\psi_{k-1}(x)}{z-x}\,\frac{d\sigma_{k}(x)}{L_{\textbf{n},k+1}(x)}=\frac{\psi_{k}(z)}{L_{\textbf{n},k+1}(z)}.
\end{equation}
Similarly, since $\deg L_{\textbf{n},k+1}\leq N_{\textbf{n},k}$, from \eqref{newexp1} we obtain
\[
L_{\textbf{n},k+1}(z)\int\frac{\psi_{k-1}(x)}{z-x}\,\frac{d\sigma_{k}(x)}{L_{\textbf{n},k+1}(x)}=\int\frac{\psi_{k-1}(x)}{z-x}\,d\sigma_{k}(x)
\]
which combined with \eqref{recformpsi} implies
\begin{equation}\label{iterpsiform}
\psi_{k}(z)=\int\frac{\psi_{k-1}(x)}{z-x}\,d\sigma_{k}(x),\qquad k=1, 2.
\end{equation}

We consider now the functions
\[
\varphi_{k}(z):=\int\frac{L_{\textbf{n},1}(x)}{z-x}\,d s_{k}(x),\qquad k=1, 2.
\]
Observe that $\varphi_{1}=\psi_{1}$ and
\[
\varphi_{2}(z)=\iint\frac{L_{\mathbf{n},1}(x)}{(z-x)(x-t)}d\sigma_{1}(x)\,d\sigma_{2}(t),\qquad \psi_{2}(z)=\iint\frac{L_{\mathbf{n},1}(x)}{(z-t)(t-x)}d\sigma_{1}(x)\,d\sigma_{2}(t),
\]
hence
\begin{equation}\label{relphi2psi2}
\varphi_{2}(z)=\widehat{\sigma}_{2}(z)\,\psi_{1}(z)-\psi_{2}(z)
\end{equation}
where
\[
\widehat{\sigma}_{2}(z)=\int\frac{d\sigma_{2}(t)}{z-t}.
\]
We conclude from \eqref{estinfpsik} and \eqref{relphi2psi2} that
\[
\varphi_{k}(z)=O\left(\frac{1}{z^{n_{k}+1}}\right),\quad z\rightarrow\infty,\qquad k=1, 2.
\]
This is equivalent to the condition that $L_{\textbf{n},1}$ is the monic multiple orthogonal polynomial with respect to $(s_{1},s_{2})$ and $\textbf{n}$, thus $L_{\textbf{n},1}$ is unique. It then follows from \eqref{iterpsiform} and $\psi_{0}=L_{\textbf{n},1}$ that
the functions $\psi_{k}$ are uniquely determined. From \eqref{defpsik} we deduce that the roots of $L_{\textbf{n},2}$ are the zeros of $\psi_{1}$ in its region of analyticity, so the polynomial $L_{\textbf{n},2}$ is also uniquely determined.
\end{proof}

\begin{definition} \label{defKkappa}
Given the polynomials ${\mathbf{Q}}_{\bf n} = (Q_{{\bf n},1}, Q_{{\bf n},2})$ associated with $\mathcal{N}(\sigma_1,\sigma_2)$ and $\mathbf{n}$, set
\begin{equation}\label{defKnk}
K_{{\bf n},0}=1, \qquad K_{{\bf n},k}= \left|  \int {Q}^2_{{\bf n},k}(x)\frac{\mathcal{H}_{{\bf n},k}(x)\, d\sigma_k(x)}{ {Q}_{{\bf n},k-1}(x)\, {Q}_{{\bf n},k+1}(x)}\right|^{-1/2} \qquad k=1, 2,
\end{equation}
where ${Q}_{{\bf n},0}\equiv{Q}_{{\bf n},3}\equiv1$. Take
\begin{equation}\label{def:kappas}
\kappa_{{\bf n},k}  = \frac{K_{{\bf n},k} }{K_{{\bf n},k-1}}, \qquad k=1, 2.
\end{equation}
Define
\begin{align}
{q}_{{\bf n},k} & = \kappa_{{\bf n},k}\,{Q}_{{\bf n},k}, \qquad k=1, 2,\label{def:littleqnk}\\
h_{{\bf n},k}(z) & = K_{{\bf n},k-1}^2 \mathcal{H}_{{\bf n},k}(z), \qquad 1\leq k\leq 3.\label{def:littlehnk}
\end{align}
\end{definition}

With this notation, $ {q}_{{\bf n},k}$ is the orthonormal polynomial of degree $N_{\textbf{n},k}$ with respect to the varying measure
\begin{equation}\label{varyingmeasqnk}
{|h_{{\bf n},k}(x)|\,d\sigma_k(x)}/{|{Q}_{{\bf n},k-1}(x)\,{Q}_{{\bf n},k+1}(x)}|
\end{equation}
and $ {Q}_{{\bf n},k}$ is the corresponding monic orthogonal polynomial. Note that the measure
\begin{equation}\label{varyingmeasqnk:2}
{h_{{\bf n},k}(x)\,d\sigma_k(x)}/{{Q}_{{\bf n},k-1}(x)\,{Q}_{{\bf n},k+1}(x)}
\end{equation}
has constant sign on the interval $\Delta_{k}$. We define $\varepsilon_{\textbf{n},k}=\pm 1$ as the sign of this measure.

The following result is a consequence of Proposition 3.1 in \cite{LL}. It can also be derived from Lemma \ref{limitehk2} below.

\begin{lemma}\label{limitehk} Suppose that $\sigma'_k > 0$ a.e. on $\Delta_k=\mathrm{Co}(\supp(\sigma_{k}))=[a_{k},b_{k}]$, $k=1, 2$. Let $\Lambda \subset \Z_+^2(\circledast)$ be an infinite sequence of distinct multi-indices such that there exists a constant $c\geq 0$ for which $n_1 \leq n_{2} + c$, for all ${\bf n}=(n_{1},n_{2})\in\Lambda$. We have
\begin{equation}\label{limenkhnk}
\lim_{{\bf n} \in \Lambda} \varepsilon_{\mathbf{n},k-1}\,h_{{\bf n},k}( z) = \frac{1}{\sqrt{(z-a_{k-1})(z-b_{k-1})}}, \qquad k=2, 3,
\end{equation}
uniformly on compact subsets of $\overline{\C} \setminus \Delta_{k-1}$. The branch of the square root is taken so that $\sqrt{x}>0$ for $x>0$.
\end{lemma}

We will use some known results on the relative asymptotics of orthogonal polynomials with respect to varying measures in the scalar case. For convenience of the reader we state the result that we need with a degree of generality sufficient in our context. The statement of Lemma~\ref{relescalar} below is a corollary of \cite[Theorem 2]{berndolyo}. If $(p_{n})_{n \in \Lambda}$ is a sequence of polynomials, we say that the zeros of $(p_{n})_{n\in \Lambda}$ are uniformly bounded away from a compact set $K\subset\mathbb{C}$ if there exists $\delta>0$ such that $\textrm{dist}(Z(p_{n}),K)>\delta$ for all $n\in \Lambda$, where $Z(p_{n})$ is the set of zeros of $p_{n}$ and $\textrm{dist}(\cdot,\cdot)$ denotes the Euclidean distance between the indicated sets.

\begin{lemma}\label{relescalar} Assume that:
\begin{itemize}
\item[i)] $\mu$ is a finite positive Borel measure supported on a compact interval $\Delta$ of the real line with $\mu' > 0$ a.e. on $\Delta$,
\item[ii)] $\rho$  is a non-negative $\mu$-integrable function on $\Delta$ such that $|q|\rho^{\pm 1} \in L^{\infty}(\mu)$ where $q$ is a non-zero polynomial,
\item[iii)]  $(w_{2n})_{n\geq 0}$, $\deg w_{2n} \leq 2n$, is a sequence of polynomials with real coefficients whose zeros are uniformly bounded away from $\Delta$,
\item[iv)] $(g_{n})_{n\geq 0}, (\widetilde{g}_{n})_{n\geq 0}$ are sequences of non-negative continuous functions on $\Delta$
which converge uniformly on $\Delta$ to positive functions $g$ and $\widetilde{g}$, respectively.
\end{itemize}
Let $p_{n,l}$ and $\widetilde{p}_{n,l}$ be the orthonormal polynomials of degree $l$ with respect to the varying measures $g_n\,d\mu/|w_{2n}|$ and $\rho\,\widetilde{g}_n\,g_n\,d\mu/|w_{2n}|$, respectively. Then, for every $r\in\mathbb{Z}$ fixed, we have
\[ \lim_{n\rightarrow\infty} \frac{\widetilde{p}_{n,n+r}(z)}{p_{n,n+r}(z)}=\mathsf{G}_{\Delta}(\widetilde{g}\,\rho;z)
\]
uniformly on compact subsets of $\overline{\C} \setminus \Delta$.
\end{lemma}

Note that in \cite[Theorem 2]{berndolyo} it is required that $|q|\rho^{\pm 1} \in L^{\infty}(dx)$, but this condition follows from $i)$ and $ii)$.

In the rest of this section we assume that the following conditions hold for each $k=1, 2$:
\begin{itemize}
\item[c1)] $\sigma_k'>0$ a.e. on $\Delta_k=\supp(\sigma_{k})= [a_{k},b_{k}]$;
\item[c2)] $\rho_{k}$ is a non-negative function on $\Delta_{k}$ such that $\rho_{k}\in L^{1}(\sigma_{k})$,
and there exists a non-zero polynomial $q$ such that  $|q|\rho_k^{\pm 1} \in L^{\infty}(\sigma_k)$.
\end{itemize}
These conditions imply in particular that $\rho_{k}>0$ a.e. on $\Delta_{k}$. Associated with the Nikishin system $(\widetilde{s}_{1},\widetilde{s}_{2}) = {\mathcal{N}}(\widetilde{\sigma}_1,\widetilde{\sigma}_2)$, where $d\widetilde{\sigma}_k = \rho_k\,d\sigma_k$, we have the corresponding monic orthogonal polynomials $\widetilde{Q}_{{\bf n},k}$, the orthonormal ones $\widetilde{q}_{{\bf n},k}$, orthonormalizing constants $\widetilde{\kappa}_{{\bf n},k}, \widetilde{K}_{{\bf n},k}$, and functions $\widetilde{\mathcal{H}}_{{\bf n},k},\widetilde{h}_{{\bf n},k}$
defined as the analogous ones without tilde. In the course of our study we will obtain the limits
\[ \lim_{{\bf n} \in \Lambda} \frac{\widetilde{Q}_{{\bf n},k}}{{Q}_{{\bf n},k}},\qquad  \lim_{{\bf n} \in \Lambda} \frac{\widetilde{q}_{{\bf n},k}}{{q}_{{\bf n},k}},\qquad \lim_{{\bf n} \in \Lambda} \frac{\widetilde{\kappa}_{{\bf n},k}}{{\kappa}_{{\bf n},k}},\qquad \lim_{{\bf n} \in \Lambda}\frac{\widetilde{\Psi}_{{\bf n},k}}{\Psi_{{\bf n},k}}
\]
where $\Lambda \subset \mathbb{Z}_+^2(\circledast)$ is an infinite sequence of distinct multi-indices.

For the study of this problem we adapt a method devised by A.I. Aptekarev \cite{sasha} to analyze the strong asymptotics of type II multiple orthogonal polynomials of a Nikishin system using fixed-point theorems. The great disadvantage of the relations \eqref{ortogonalidades}  is that the  $Q_{\mathbf{n},k}$ appear simultaneously as the orthogonal polynomials and in the varying part of the measures of orthogonality with the subindices displaced. We will temporarily detach this connection introducing an appropriate mapping.

For notational convenience, throughout the rest of this paper we set $\Delta_{0}=\Delta_{3}=\emptyset$.

Given $\textbf{n}\in\mathbb{Z}_{+}^{2}(\circledast)$, let $\mathcal{P}_{{\bf n},k}$, $k=1, 2,$ denote the collection of all polynomials of degree at most $N_{\textbf{n},k}=\sum_{j=k}^{2} n_{j}$ with real coefficients and zeros in $\mathbb{C}\setminus(\Delta_{k-1}\cup\Delta_{k+1})$. Note that the zero polynomial is excluded from $\mathcal{P}_{\mathbf{n},k}$. Let
\[\mathcal{P}_{\bf n} := \mathcal{P}_{{\bf n},1} \times \mathcal{P}_{{\bf n},2}.
\]

\begin{definition} \label{Tn}
Define
 \begin{equation}\label{defmapTn}
 T_{\bf  n} :\mathcal{P}_{\bf n} \rightarrow \mathcal{P}_{\bf n}
 \end{equation}
where $T_{\bf n}(\mathbf{\widetilde{Q}}) = \mathbf{\widehat{Q}},$ $\mathbf{\widetilde{Q}}= (\widetilde{Q}_1,\widetilde{Q}_2),$ and $\mathbf{\widehat{Q}} = (\widehat{Q}_1,\widehat{Q}_2)$ is the unique vector of monic polynomials that satisfies
 \begin{equation} \label{ortogtilde}
 \int x^\nu \widehat{Q}_k(x) \frac{H_k(\mathbf{\widetilde{Q}},x)\,\rho_k(x)\,d\sigma_k(x)}{\widetilde{Q}_{k-1}(x)\,\widetilde{Q}_{k+1}(x)} = 0, \qquad \nu = 0,\ldots,N_{\mathbf{n},k}-1, \qquad k=1, 2,
 \end{equation}
 with
\begin{equation}\label{defHqtilde}
H_{k+1}(\mathbf{\widetilde{Q}},z) := \int\frac{(\widehat{Q}_{k}(x))^{2}}{z-x} \frac{H_k(\mathbf{\widetilde{Q}},x)\,\rho_k(x)\,d\sigma_k(x)}{\widetilde{Q}_{k-1}(x)\,\widetilde{Q}_{k+1}(x)}, \qquad k=1, 2, \qquad H_1(\mathbf{\widetilde{Q}},z) \equiv 1,
\end{equation}
and, by convention, $\widetilde{Q}_0\equiv\widetilde{Q}_{3} \equiv 1$. Note that $\deg \widehat{Q}_{k}=N_{\mathbf{n},k}$ for each $k=1, 2$.
\end{definition}

Observe that with $H_1(\mathbf{\widetilde{Q}},z)$ one can find $\widehat{Q}_1$ and $H_2(\mathbf{\widetilde{Q}},z)$.
Then, one can construct
$\widehat{Q}_2$ and $H_3(\mathbf{\widetilde{Q}},z)$.  For each $k=1, 2,$ the function $ {H_k(\mathbf{\widetilde{Q}},x)\,\rho_{k}(x)}/{\widetilde{Q}_{k-1}(x)\,\widetilde{Q}_{k+1}(x)}$ has constant sign on $\Delta_k$ and, therefore, $\widehat{Q}_k$ of degree $N_{\textbf{n},k}$ is uniquely determined.  Therefore, $T_{\bf n}$ is well defined.

\begin{theorem}\label{puntofijo}
 $T_{\mathbf{n}}$ has a unique fixed point in $\mathcal{P}_{\bf n}$ which coincides with $(\widetilde{Q}_{\mathbf{n},1},\widetilde{Q}_{\mathbf{n},2})$, the system of monic multiple orthogonal polynomials associated with the multi-index $\mathbf{n}$ and the Nikishin system $\mathcal{N}(\widetilde{\sigma}_{1},\widetilde{\sigma}_{2})$.
\end{theorem}

\begin{proof} Assume that $T_{\mathbf{n}}(\mathbf{\widetilde{Q}})=\mathbf{ \widetilde{Q}}$. From the definition of the operator $T_{\mathbf{n}}$, the $k$-th component of the vector $\mathbf{\widetilde{Q}}$ must be a monic polynomial of degree $N_{\mathbf{n},k}$, for each $k=1, 2$. The conditions \eqref{ortogtilde}--\eqref{defHqtilde} reduce to \eqref{ortogonalidades}--\eqref{entreHnks} for the multi-index $\bf n$
with respect to $\mathcal{N}(\widetilde{\sigma}_1,\widetilde{\sigma}_2)$.
Thus, by Theorem~ \ref{unicidad} applied to $\mathcal{N}(\widetilde{\sigma}_1,\widetilde{\sigma}_2)$, we obtain $\mathbf{\widetilde{Q}} = (\widetilde{Q}_{{\bf n},1},\widetilde{Q}_{{\bf n},2})$.
\end{proof}

In analogy with Definition \ref{defKkappa}, we introduce a new normalization appropriate with the use of the operator $T_{\bf n}$.

\begin{definition} \label{normaTn}
Let $\mathbf{\widetilde{Q}}=(\widetilde{Q}_1,\widetilde{Q}_2)\in\mathcal{P}_{\bf n}$, ${\bf n}\in \Z_+^2(\circledast)$, and $T_{\bf n}(\mathbf{\widetilde{Q}})= \mathbf{\widehat{Q}}=(\widehat{Q}_1,\widehat{Q}_2)$. Set
\begin{equation}\label{defKQs}
K_{0}(\mathbf{\widetilde{Q}}):=1, \qquad K_{k}(\mathbf{\widetilde{Q}}):=\left|\int (\widehat{Q}_{k}(x))^2\,\frac{H_k(\mathbf{\widetilde{Q}},x)\,\rho_k(x)\,d\sigma_k(x)}
{\widetilde{Q}_{k-1}(x)\,\widetilde{Q}_{k+1}(x)}\right|^{-1/2}, \qquad k=1, 2,
\end{equation}
where the functions $H_k(\mathbf{\widetilde{Q}},x)$ are defined in \eqref{defHqtilde}, and $\widetilde{Q}_{0}\equiv \widetilde{Q}_{3}\equiv 1$. Take
\begin{equation} \label{def:kappas2} \kappa_{k}(\mathbf{\widetilde{Q}}):=\frac{K_{k}(\mathbf{\widetilde{Q}})}{K_{k-1}(\mathbf{\widetilde{Q}})},\qquad k=1, 2.
\end{equation}
Define
\begin{align*}
\widehat{q}_k & := \kappa_{k}(\mathbf{\widetilde{Q}})\,\widehat{Q}_k, \qquad k=1, 2,\\
h_{k}(\mathbf{\widetilde{Q}},z) & := K_{k-1}^2(\mathbf{\widetilde{Q}})\,H_{k}(\mathbf{\widetilde{Q}},z), \qquad 1\leq k \leq 3.
\end{align*}
\end{definition}

Observe that $h_{1}(\mathbf{\widetilde{Q}},z)\equiv 1$. With this notation, $\widehat{q}_k$ is the orthonormal polynomial of degree $N_{\mathbf{n},k}$ with respect to the measure
\[ { |h_k(\mathbf{\widetilde{Q}},x)|\, \rho_k(x)\,d\sigma_k(x)}/{|\widetilde{Q}_{k-1}(x)\,\widetilde{Q}_{k+1}(x)|}\]
and $\widehat{Q}_k$ is the corresponding monic orthogonal polynomial.

\begin{lemma}\label{limitehk2} Let $(\widetilde{\mathbf{L}}_{\bf n})_{{\bf n} \in \Lambda}$, $\widetilde{\mathbf{L}}_{\bf n}=(\widetilde{L}_{\mathbf{n},1},\widetilde{L}_{\mathbf{n},2})\in \mathcal{P}_{\bf n}$, where $\Lambda \subset \Z_+^2(\circledast)$ is an infinite sequence of distinct multi-indices such that there exists a constant $c\geq 0$ for which $n_1\leq n_{2} + c$, for all ${\bf n}=(n_{1},n_{2}) \in \Lambda$. We also assume that for each $k=1, 2$, the zeros of $(\widetilde{L}_{\mathbf{n},k})_{\mathbf{n}\in\Lambda}$ are uniformly bounded away from $\Delta_{k-1}\cup\Delta_{k+1}$. Let $\widetilde{\varepsilon}_{\mathbf{n},k}$ denote the sign of the measure
\[
h_k(\mathbf{\widetilde{L}}_{\bf n},x)\,\rho_{k}(x)\,d\sigma_{k}(x)/\widetilde{L}_{\mathbf{n},k-1}(x)\,\widetilde{L}_{\mathbf{n},k+1}(x)
\]
on $\Delta_{k}$. Then
\begin{equation}\label{convhkfunc}
\lim_{{\bf n} \in \Lambda} \widetilde{\varepsilon}_{\mathbf{n},k-1}\,h_k(\mathbf{\widetilde{L}}_{\bf n},z)= \frac{1}{\sqrt{(z-a_{k-1})(z-b_{k-1})}}, \qquad k=2, 3,
\end{equation}
uniformly on compact subsets of $\overline{\C} \setminus \Delta_{k-1}$.
\end{lemma}

\begin{proof} When $k=2$, we have
\begin{equation}\label{intreph2}
h_{2}(\mathbf{\widetilde{L}}_{\bf n},z)= \int   \frac{(\widehat{l}_{{\bf n},1}(x))^2}{z-x} \frac{\rho_1(x)\,d\sigma_1(x)}{ \widetilde{L}_{{\bf n},2}(x)},
 \end{equation}
 where $\widehat{l}_{{\bf n},1}$ is the orthonormal polynomial of degree $|\bf n|$ with respect to the varying measure ${\rho_1\, d\sigma_1}/{ |\widetilde{L}_{{\bf n},2}|}$. Since the zeros of $\widetilde{L}_{{\bf n},2}$ are uniformly bounded away from $\Delta_{1}$ and $\deg \widetilde{L}_{{\bf n},2} \leq 2\deg \widehat{l}_{{\bf n},1}$, applying Theorem 8 in \cite{BerLop} we obtain
 \begin{equation}\label{weakconv:1}
 \lim_{\mathbf{n}\in\Lambda}\int f(x)\,(\widehat{l}_{{\bf n},1}(x))^2\,\frac{\rho_{1}(x)\,d\sigma_{1}(x)}{|\widetilde{L}_{{\bf n},2}(x)|}=
 \frac{1}{\pi}\int_{\Delta_{1}}f(x)\,\frac{dx}{\sqrt{(b_1-x)(x-a_1)}}
 \end{equation}
 for any function $f$ continuous on $\Delta_1$. Taking $f(x)=(z-x)^{-1}$ in \eqref{weakconv:1} and using \eqref{intreph2} we obtain
 \begin{equation}\label{asymph2}
 \lim_{\mathbf{n}\in\Lambda}\widetilde{\varepsilon}_{\mathbf{n},1}\,h_{2}(\mathbf{\widetilde{L}}_{\bf n},z)=
 \frac{1}{\pi}\int_{\Delta_{1}}\frac{1}{z-x}\,\frac{dx}{\sqrt{(b_1-x)(x-a_1)}}=\frac{1}{\sqrt{(z-a_1)(z-b_1)}}
 \end{equation}
for any $z\in\overline{\mathbb{C}}\setminus\Delta_1$. The convergence is in fact uniform on compact subsets of $\overline{\mathbb{C}}\setminus\Delta_1$ since the integral functions are uniformly bounded on compact subsets of the indicated region.

We have
\[
h_{3}(\mathbf{\widetilde{L}}_{\bf n},z)=\int\frac{(\widehat{l}_{{\bf n},2}(x))^2}{z-x} \frac{h_{2}(\mathbf{\widetilde{L}}_{\bf n},x)}{ \widetilde{L}_{{\bf n},1}(x)}\,\rho_2(x)\,d\sigma_2(x)
\]
where $\widehat{l}_{{\bf n},2}$ is the orthonormal polynomial of degree $N_{\mathbf{n},2}$ with respect to the positive measure
\[
\frac{|h_{2}(\mathbf{\widetilde{L}}_{\bf n},x)|}{|\widetilde{L}_{{\bf n},1}(x)|}\rho_2(x)\,d\sigma_2(x).
\]
By \eqref{asymph2} we have
\[
\lim_{{\bf n} \in \Lambda} |h_2(\mathbf{\widetilde{L}}_{\bf n},x)| = \frac{1}{|\sqrt{(x-a_{1})(x-b_{1})}|}
\]
uniformly on $\Delta_{2}$, the zeros of $\widetilde{L}_{{\bf n},1}(x)$ are uniformly bounded away from $\Delta_{2}$, and
\[
\deg \widetilde{L}_{{\bf n},1}\leq N_{\mathbf{n},1}\leq 2N_{\mathbf{n},2}+c\leq 2\,(\deg(\widehat{l}_{{\bf n},2})+c).
\]
So Theorem 8 in \cite{BerLop} can be applied in this context if we identify in the notation of that theorem the measure $\mu_{n}$ with $|h_{2}(\mathbf{\widetilde{L}}_{\bf n},x)|\,\rho_2(x)\,d\sigma_2(x)$, the function $w_{2n}$ with $|\widetilde{L}_{{\bf n},1}|$, and $l_{n,n+k}$ with $\widehat{l}_{{\bf n},2}$. Applying that theorem we obtain
\[
\lim_{\mathbf{n}\in\Lambda}\int f(x)(\widehat{l}_{{\bf n},2}(x))^2\frac{|h_{2}(\mathbf{\widetilde{L}}_{\bf n},x)|}{|\widetilde{L}_{{\bf n},1}(x)|}\rho_2(x)\,d\sigma_2(x)=\frac{1}{\pi}\int_{\Delta_{2}}f(x)\,\frac{dx}{\sqrt{(b_2-x)(x-a_2)}}
\]
for any continuous function $f$ on $\Delta_2$. Arguing as before, this implies
\[
\lim_{{\bf n} \in \Lambda} \widetilde{\varepsilon}_{\mathbf{n},2}\,h_{3}(\mathbf{\widetilde{L}}_{\bf n},z)=\frac{1}{\sqrt{(z-a_{2})(z-b_{2})}}
\]
uniformly on compact subsets of $\overline{\mathbb{C}}\setminus\Delta_{2}$.
\end{proof}

\begin{lemma}\label{relvectorial} Assume that conditions $c1)$ and $c2)$ hold for each $k=1, 2$. Let $(\widetilde{\mathbf{L}}_{\bf n})_{{\bf n} \in \Lambda}$, $\widetilde{\mathbf{L}}_{\bf n}=(\widetilde{L}_{\mathbf{n},1},\widetilde{L}_{\mathbf{n},2})\in \mathcal{P}_{\bf n}$, where $\Lambda \subset \Z_+^2(\circledast)$ is an infinite sequence of distinct multi-indices such that there exists a constant $c\geq 0$ for which $n_1 \leq n_{2} + c$, for all ${\bf n}=(n_{1},n_{2}) \in \Lambda$. Assume that for each $k$, the zeros of $(\widetilde{L}_{\mathbf{n},k})_{\mathbf{n}\in\Lambda}$ are uniformly bounded away from $\Delta_{k-1}\cup\Delta_{k+1}$. Suppose that
\begin{equation} \label{fk} \lim_{{\bf n} \in \Lambda} \frac{\widetilde{L}_{\mathbf{n},k}}{Q_{\mathbf{n},k}}=f_k, \qquad k=1, 2,
\end{equation}
uniformly on $\Delta_{k-1} \cup \Delta_{k+1}$, where $(Q_{\mathbf{n},1},Q_{\mathbf{n},2})$ is the system of monic multiple orthogonal polynomials associated with the Nikishin system $\mathcal{N}(\sigma_1,\sigma_2)$ and ${\bf n} \in \Lambda$. We assume that $f_{k}>0$ on $\Delta_{k-1} \cup \Delta_{k+1}$. Finally, let $T_{\bf n}(\widetilde{\mathbf{L}}_{\bf n}) = \widehat{\mathbf{L}}_{\bf n} = (\widehat{ {L}}_{{\bf n},1},\widehat{ {L}}_{{\bf n},2})$ and $\widehat{ {l}}_{{\bf n},k} = \kappa_{k}(\mathbf{\widetilde{L}}_{\bf n})\widehat{ {L}}_{{\bf n},k}$. Then, for each $k=1, 2$,
\begin{equation} \label{limfund3}
\lim_{{\bf n} \in \Lambda} \frac{\widehat{l}_{{\bf n},k}(z)}{ { {q}}_{{\bf n},k}(z)}=\mathsf{G}_{\Delta_{k}}(\rho_k/f_{k-1}f_{k+1};z), \qquad
\lim_{{\bf n} \in \Lambda} \frac{\kappa_{k}(\mathbf{\widetilde{L}}_{\bf n})}{\kappa_{{\bf n},k}}=\mathsf{G}_{\Delta_{k}}(\rho_k/f_{k-1}f_{k+1};\infty),
\end{equation}
uniformly on compact subsets of $\overline{\C} \setminus \Delta_k$ ($f_0 \equiv f_{3} \equiv 1$).
\end{lemma}

\begin{proof} Fix $k=1, 2$. By definition we know that $\widehat{l}_{{\bf n},k}$ is the orthonormal polynomial of degree $N_{{\bf{n}},k}$ with respect to the varying measure
\[
\frac{|h_k(\mathbf{\widetilde{L}}_{\bf n},x)|\,\rho_k(x)\,d\sigma_k(x)}{|\widetilde{L}_{{\bf n},k-1}(x)\,\widetilde{L}_{{\bf n},k+1}(x)|} =\frac{| {Q}_{{\bf n},k-1}(x)\, {Q}_{{\bf n},k+1}(x)|}{|\widetilde{L}_{{\bf n},k-1}(x)\,\widetilde{L}_{{\bf n},k+1}(x)|} \frac{|h_k(\mathbf{\widetilde{L}}_{\bf n},x)|\,\rho_k(x)}{|h_{{\bf n},k}(x)|}
 \frac{|h_{{\bf n},k}(x)|\,d\sigma_k(x)}{ | {Q}_{{\bf n},k-1}(x)\,{Q}_{{\bf n},k+1}(x)|}.
\]
Using \eqref{fk}, Lemma \ref{limitehk}, and Lemma \ref{limitehk2}, we obtain
\[
\lim_{{\bf n}\in \Lambda} \frac{|{Q}_{{\bf n},k-1}(x) {Q}_{{\bf n},k+1}(x)|}{|\widetilde{L}_{{\bf n},k-1}(x)\widetilde{L}_{{\bf n},k+1}(x)|} \frac{|h_k(\mathbf{\widetilde{L}}_{\bf n},x)| }{|h_{{\bf n},k}(x)|} = \frac{1 }{f_{k-1}(x) f_{k+1}(x)}
\]
uniformly on $\Delta_k$. The polynomial $q_{{\bf n},k}$ is the orthonormal polynomial of degree $N_{{\bf{n}},k}$ with respect to the varying measure
\[
\frac{|h_{{\bf n},k}(x)|\,d\sigma_k(x)}{ | {Q}_{{\bf n},k-1}(x)\,{Q}_{{\bf n},k+1}(x)|}.
\]
 Using Lemma \ref{relescalar}, the first limit readily follows. The second limit is a consequence of the first one applied at $\infty$.
\end{proof}

\section{Proof of Theorem~\ref{theomain}}\label{section:proof}

Before we proceed to the proof of Theorem~\ref{theomain}, we establish some preliminary facts and introduce some important definitions and notations.

Let $\Lambda\subset\mathbb{Z}_{+}^{2}(\circledast)$ be an infinite sequence of distinct multi-indices $\mathbf{n}=(n_{1},n_{2})$ such that $\sup_{\mathbf{n}\in\Lambda}(n_{1}-n_{2})<\infty$, that is, there exists $c>0$ such that $n_{1}\leq n_{2}+c$.

Let $\Phi=(\Phi_{1},\Phi_{2})\in\mathbf{C}_{\Delta}^{+}$ be the fixed point of the mapping $T$. This means
\[
\Phi_{k}(x)=\mathsf{G}_{\Delta_{k}}(\rho_{k}/\Phi_{k-1}\Phi_{k+1};x),\qquad x\in\Delta_{k-1}\cup\Delta_{k+1},\quad k=1, 2,
\]
($\Phi_{0}\equiv \Phi_{3}\equiv 1$). The analytic extension of $\Phi_{k}$ to $\overline{\mathbb{C}}\setminus\Delta_{k}$ will be denoted again by $\Phi_{k}$. For $\mathbf{n}\in\Lambda$ we define the space
\[
S_{\mathbf{n}}^{+}=\left\{\left(\frac{ P_{\mathbf{n},1}}{Q_{\mathbf{n},1}},\frac{ P_{\mathbf{n},2}}{Q_{\mathbf{n},2}}\right): P_{\mathbf{n},k}\in\mathcal{P}_{\mathbf{n},k},\,\,\frac{P_{\mathbf{n},k}}{Q_{\mathbf{n},k}}>0\,\,\mbox{on}\,\,\Delta_{k-1}\cup\Delta_{k+1}\right\}.
\]
It is clear that $S_{\mathbf{n}}^{+}\subset \mathbf{C}_{\Delta}^{+}$. On $S_{\mathbf{n}}^{+}$ we define another operator.

\begin{definition} \label{operTn}
Let
\[ \widetilde{T}_{\mathbf{n}}:S_{\mathbf{n}}^{+}\rightarrow S_{\mathbf{n}}^{+}\]
be the operator defined by
\begin{equation}\label{defTntilde}
\widetilde{T}_{\mathbf{n}}\left( \frac{P_{\mathbf{n},1}}{ Q_{\mathbf{n},1}},\frac{ P_{\mathbf{n},2}}{ Q_{\mathbf{n},2}} \right)=\left(\frac{\kappa_1(\mathbf{P}_{\bf n})\,\widehat{P}_{\mathbf{n},1}}{\kappa_{\mathbf{n},1}\,Q_{\mathbf{n},1}},\frac{ \kappa_2(\mathbf{P}_{\bf n})\,\widehat{P}_{\mathbf{n},2}}{\kappa_{\mathbf{n},2}\,Q_{\mathbf{n},2}}\right),
\end{equation}
where $(\widehat{P}_{\mathbf{n},1},\widehat{P}_{\mathbf{n},2})=T_{\mathbf{n}}(\mathbf{P}_{\bf n})$ and
$\mathbf{P}_{\bf n} = (P_{\mathbf{n},1},P_{\mathbf{n},2})$.  The constants $\kappa_{{\bf n},k}$  are defined in \eqref{def:kappas} and the $\kappa_k(\mathbf{P}_n)$ in \eqref{def:kappas2}.
\end{definition}

\begin{theorem}\label{fixedpointTntilde}
The operator $\widetilde{T}_{\mathbf{n}}:S_{\mathbf{n}}^{+}\rightarrow S_{\mathbf{n}}^{+}$ has a unique fixed point in $S_{\mathbf{n}}^{+}$, which is given by the vector
\begin{equation}\label{vectorfixedpoint}
\left(c_{\mathbf{n},1}\frac{\widetilde{Q}_{\mathbf{n},1}}{Q_{\mathbf{n},1}},c_{\mathbf{n},2}\frac{\widetilde{Q}_{\mathbf{n},2}}{Q_{\mathbf{n},2}}\right),
\end{equation}
where
$
c_{\mathbf{n},1}=\left(\widetilde{\kappa}_{\mathbf{n},1}/\kappa_{\mathbf{n},1}\right)^{4/3}\left(\widetilde{\kappa}_{\mathbf{n},2}/\kappa_{\mathbf{n},2}\right)^{2/3}$ and  $c_{\mathbf{n},2}=\left(\widetilde{\kappa}_{\mathbf{n},1}/\kappa_{\mathbf{n},1}\right)^{2/3}\left(\widetilde{\kappa}_{\mathbf{n},2}/\kappa_{\mathbf{n},2}\right)^{4/3}
$.
\end{theorem}
\begin{proof}
Recall the definition of the constants $\widetilde{\kappa}_{\mathbf{n},k}, \widetilde{K}_{\mathbf{n},k}$ and the functions $\widetilde{\mathcal{H}}_{\mathbf{n},k}$ associated with the Nikishin system $\mathcal{N}(\widetilde{\sigma}_{1},\widetilde{\sigma}_{2})$. Let $\mathbf{P}_{\mathbf{n}}=(c_{1}\,\widetilde{Q}_{\mathbf{n},1},c_{2}\,\widetilde{Q}_{\mathbf{n},2})$ for some constants $c_{1}, c_{2}>0$. By the orthogonality conditions that characterize the polynomials $\widetilde{Q}_{\mathbf{n},1}$ and $\widetilde{Q}_{\mathbf{n},2}$, it is clear that
\begin{equation}\label{invTn}
T_{\mathbf{n}}(\mathbf{P}_{\mathbf{n}})=(\widetilde{Q}_{\mathbf{n},1},\widetilde{Q}_{\mathbf{n},2}).
\end{equation}
So a vector of the form
\begin{equation}\label{fpcand}
\left(c_{1}\frac{\widetilde{Q}_{\mathbf{n},1}}{Q_{\mathbf{n},1}},c_{2}\frac{\widetilde{Q}_{\mathbf{n},2}}{Q_{\mathbf{n},2}}\right), \qquad c_{1}, c_{2}>0,
\end{equation}
is a fixed point of $\widetilde{T}_{\mathbf{n}}$ if and only if
\[
\left(\frac{\kappa_{1}(\mathbf{P}_{\mathbf{n}})\,\widetilde{Q}_{\mathbf{n},1}}{\kappa_{\mathbf{n},1}\,Q_{\mathbf{n},1}},\frac{\kappa_{2}(\mathbf{P}_{\mathbf{n}})\,\widetilde{Q}_{\mathbf{n},2}}{\kappa_{\mathbf{n},2}\,Q_{\mathbf{n},2}}\right)=\left(c_{1}\frac{\widetilde{Q}_{\mathbf{n},1}}{Q_{\mathbf{n},1}},c_{2}\frac{\widetilde{Q}_{\mathbf{n},2}}{Q_{\mathbf{n},2}}\right),
\]
equivalently,
\begin{equation}\label{condcks}
c_{k}=\frac{\kappa_{k}(\mathbf{P}_{\mathbf{n}})}{\kappa_{\mathbf{n},k}},\qquad k=1, 2.
\end{equation}
Now, from \eqref{invTn}, \eqref{defKQs}, and \eqref{def:kappas2} we deduce that $K_{1}(\mathbf{P}_{\mathbf{n}})=c_{2}^{1/2}\,\widetilde{K}_{\mathbf{n},1}$ and so
\[
\kappa_{1}(\mathbf{P}_{\mathbf{n}})=K_{1}(\mathbf{P}_{\mathbf{n}})=c_{2}^{1/2}\,\widetilde{K}_{\mathbf{n},1}=c_{2}^{1/2}\,\widetilde{\kappa}_{\mathbf{n},1}.
\]
This implies by \eqref{condcks} that
\[
\frac{c_{1}}{c_{2}^{1/2}}=\frac{\widetilde{\kappa}_{\mathbf{n},1}}{\kappa_{\mathbf{n},1}}.
\]
From \eqref{invTn} and \eqref{defHqtilde} we obtain that $H_{2}(\mathbf{P}_{\mathbf{n}},z)=c_{2}^{-1}\,\widetilde{\mathcal{H}}_{\mathbf{n},2}(z)$, and so by \eqref{defKQs} we have
\[
K_{2}(\mathbf{P}_{\mathbf{n}})=(c_{1}\,c_{2})^{1/2}\widetilde{K}_{\mathbf{n},2}.
\]
Hence
\[
\kappa_{2}(\mathbf{P}_{\mathbf{n}})=\frac{K_{2}(\mathbf{P}_{\mathbf{n}})}{K_{1}(\mathbf{P}_{\mathbf{n}})}=c_{1}^{1/2} \frac{\widetilde{K}_{\mathbf{n},2}}{\widetilde{K}_{\mathbf{n},1}}=c_{1}^{1/2}\,\widetilde{\kappa}_{\mathbf{n},2}.
\]
So from \eqref{condcks} we obtain
\[
\frac{c_{2}}{c_{1}^{1/2}}=\frac{\widetilde{\kappa}_{\mathbf{n},2}}{\kappa_{\mathbf{n},2}}.
\]
We have shown that \eqref{condcks} is equivalent to
\[
\frac{c_{1}}{c_{2}^{1/2}}=\frac{\widetilde{\kappa}_{\mathbf{n},1}}{\kappa_{\mathbf{n},1}},\qquad \frac{c_{2}}{c_{1}^{1/2}}=\frac{\widetilde{\kappa}_{\mathbf{n},2}}{\kappa_{\mathbf{n},2}}.
\]
This system of equations has a unique solution given by $c_{1}=\left(\widetilde{\kappa}_{\mathbf{n},1}/\kappa_{\mathbf{n},1}\right)^{4/3}\left(\widetilde{\kappa}_{\mathbf{n},2}/\kappa_{\mathbf{n},2}\right)^{2/3}$ and  $c_{2}=\left(\widetilde{\kappa}_{\mathbf{n},1}/\kappa_{\mathbf{n},1}\right)^{2/3}\left(\widetilde{\kappa}_{\mathbf{n},2}/\kappa_{\mathbf{n},2}\right)^{4/3}$.

We have shown that among vectors of the form \eqref{fpcand}, the mapping $\widetilde{T}_{\mathbf{n}}$ has a unique fixed point which is given by \eqref{vectorfixedpoint}. Now suppose that the vector
\[
\left(c_{1}\frac{P_{\mathbf{n},1}}{Q_{\mathbf{n},1}},c_{2}\frac{P_{\mathbf{n},2}}{Q_{\mathbf{n},2}}\right)\in S_{\mathbf{n}}^{+}
\]
is a fixed point of $\widetilde{T}_{\mathbf{n}}$, where $c_{1}, c_{2}\neq 0$, and $P_{\mathbf{n},k}$ is a monic polynomial of degree at most $N_{\mathbf{n},k}$ for each $k=1, 2$. If we show that $c_{k}>0$ and $P_{\mathbf{n},k}=\widetilde{Q}_{\mathbf{n},k}$, the proof will be complete.

Being a fixed point means that
\[
\left(\frac{\kappa_{1}(\mathbf{P}_{\mathbf{n}})\,\widehat{P}_{\mathbf{n},1}}{\kappa_{\mathbf{n},1}\,Q_{\mathbf{n},1}},\frac{\kappa_{2}(\mathbf{P}_{\mathbf{n}})\,\widehat{P}_{\mathbf{n},2}}{\kappa_{\mathbf{n},2}\,Q_{\mathbf{n},2}}\right)=\left(c_{1}\frac{P_{\mathbf{n},1}}{Q_{\mathbf{n},1}},c_{2}\frac{P_{\mathbf{n},2}}{Q_{\mathbf{n},2}}\right),
\]
where $\mathbf{P}_{\mathbf{n}}=(c_{1}\,P_{\mathbf{n},1}, c_{2}\,P_{\mathbf{n},2})$ and $(\widehat{P}_{\mathbf{n},1}, \widehat{P}_{\mathbf{n},2})=T_{\mathbf{n}}(\mathbf{P}_{\mathbf{n}})$. Since $P_{\mathbf{n},k}$ is monic, it follows that $c_{k}=\kappa_{k}(\mathbf{P}_{\mathbf{n}})/\kappa_{\mathbf{n},k}>0$ and $P_{\mathbf{n},k}=\widehat{P}_{\mathbf{n},k}$ for each $k=1, 2$. It is also clear that if we define the vector $\widetilde{\mathbf{P}}_{\mathbf{n}}=(P_{\mathbf{n},1}, P_{\mathbf{n},2})$, then $T_{\mathbf{n}}(\widetilde{\mathbf{P}}_{\mathbf{n}})=T_{\mathbf{n}}(\mathbf{P}_{\mathbf{n}})$. So we have shown that $T_{\mathbf{n}}(\widetilde{\mathbf{P}}_{\mathbf{n}})=\widetilde{\mathbf{P}}_{\mathbf{n}}$, i.e., $\widetilde{\mathbf{P}}_{\mathbf{n}}$ is a fixed point of $T_{\mathbf{n}}$. By Theorem~\ref{puntofijo} it follows that $\widetilde{\mathbf{P}}_{\mathbf{n}}=(\widetilde{Q}_{\mathbf{n},1},\widetilde{Q}_{\mathbf{n},2})$.
\end{proof}

For $\delta>0$, let
\[
\Delta_{k,\delta}:=\{z\in \mathbb{C}: \mathrm{dist}(z,\Delta_{k})\leq \delta\}.
\]
Throughout the proof of Theorem \ref{theomain}, we fix $\delta>0$ small enough so that
\[
\Delta_{1,\delta}\cap\Delta_{2}=\emptyset,\qquad \Delta_{2,\delta}\cap\Delta_{1}=\emptyset.
\]
Let $\Omega_{k,\delta}:=\overline{\mathbb{C}}\setminus\Delta_{k,\delta}$, let $\Omega_{\delta}:=(\Omega_{1,\delta},\Omega_{2,\delta})$, and define $H(\Omega_{\delta})$ as the space of all vectors $(g_{1},g_{2})$ with $g_{k}$ holomorphic in $\Omega_{k,\delta}$ for all $k$. Given $\mathbf{g}=(g_{1},g_{2})\in H(\Omega_{\delta})$, let
\[
\|\mathbf{g}\|_{\Omega_{\delta}}:=\max\{\sup\{|g_{k}(z)|: z\in\Omega_{k,\delta}\}: k=1, 2\}\in[0,\infty].
\]
Let $H^{*}(\Omega_{\delta})\subset H(\Omega_{\delta})$ denote the subspace consisting of all vectors $(g_{1},g_{2})\in H(\Omega_{\delta})$ such that each $g_{k}$ takes real values on $\mathbb{R}\cap\Omega_{k,\delta}$. For $\mathbf{g}\in H^{*}(\Omega_{\delta})$, let
\[
\min_{\Delta} \mathbf{g}:= \min \{\min  \{g_k(z): z\in \Delta_{k-1}\cup\Delta_{k+1}\}: k=1, 2\}.
\]
Recall that the components of the vector $\Phi=(\Phi_{1},\Phi_{2})$ are Szeg\H{o} functions on the regions $\overline{\mathbb{C}}\setminus\Delta_{k}$. Fix $C>0$ so that $C\geq 2\,\|\Phi\|_{\Omega_{\delta}}$ and $C^{-1}\leq \frac{1}{2}\min_{\Delta}\Phi$. We define the space
\[
H^{*}(\Omega_{\delta},C):=\{\mathbf{g}\in H^{*}(\Omega_{\delta}): \|\mathbf{g}\|_{\Omega_{\delta}}\leq C,\,\,\min_{\Delta} \mathbf{g}\geq C^{-1}\}.
\]
Obviously we have $\Phi\in H^{*}(\Omega_{\delta},C)$. Note also that $H^{*}(\Omega_{\delta},C)\subset \mathbf{C}^{+}_{\Delta}$.

Let $(\textbf{g}_n)_{n\in\mathbb{N}}$ be an arbitrary sequence of elements in $H^{*}(\Omega_{\delta},C)$. For each $k=1, 2$, the sequence of functions made up by the $k$-th component of $\textbf{g}_n$ is uniformly bounded on $\Omega_{k,\delta}$. Therefore, by Montel's theorem, there exists $\mathbb{I}\subset\mathbb{N}$ such that $(\textbf{g}_n)_{n\in\mathbb{I}}$ converges componentwise to some vector function $\mathbf{g}$ uniformly on compact subsets of the corresponding domains. It follows that the $k$-th component of $\mathbf{g}$ is holomorphic in $\Omega_{k,\delta}$ and real-valued on $\Omega_{k,\delta}\cap\mathbb{R}$. Additionally, it is clear that
\[
\min_{\Delta}\mathbf{g}\geq C^{-1},\qquad \|\mathbf{g}\|_{\Omega_{\delta}}\leq C,
\]
so $\mathbf{g}\in H^{*}(\Omega_{\delta},C)$. In conclusion, $H^{*}(\Omega_{\delta},C)$ is compact as a subset of $H(\Omega_{\delta})$ endowed with the metrizable topology of uniform convergence on compact sets.

For each fixed $\mathbf{n}$, the mapping $\widetilde{T}_{\mathbf{n}}$ defined in \eqref{defTntilde} is continuous with respect to the topology in $S_{\mathbf{n}}^{+}\subset H(\Omega_{\delta})$ of uniform convergence on compact sets. To justify this point, observe first that the terms $\kappa_{\mathbf{n},1}\,Q_{\mathbf{n},1},\kappa_{\mathbf{n},2}\,Q_{\mathbf{n},2}$ in the denominator in \eqref{defTntilde} remain fixed, so the continuity of $\widetilde{T}_{\mathbf{n}}$ would follow from the continuity of the mapping $T_{\mathbf{n}}$ defined in \eqref{defmapTn} as well as the continuity of the map $\widetilde{\mathbf{Q}}\in\mathcal{P}_{\mathbf{n}}\mapsto (\kappa_{1}(\widetilde{\mathbf{Q}}),\kappa_{2}(\widetilde{\mathbf{Q}}))$. Using the well-known determinantal formula that expresses an orthogonal polynomial in terms of the moments of the orthogonality measure, it easily follows from formulas \eqref{ortogtilde} and \eqref{defHqtilde} that the polynomials $(\widehat{Q}_{1},\widehat{Q}_{2})=T_{\mathbf{n}}(\widetilde{Q}_{1},\widetilde{Q}_{2})$ and the functions in \eqref{defHqtilde} depend continuously on the coefficients of the polynomials in $\widetilde{\mathbf{Q}}=(\widetilde{Q}_{1},\widetilde{Q}_{2})$, and therefore by \eqref{defKQs} and \eqref{def:kappas2} the same is true for the constants $\kappa_{1}(\widetilde{\mathbf{Q}}),\kappa_{2}(\widetilde{\mathbf{Q}})$.

\begin{proof}[\textbf{\emph{Proof of Theorem~\ref{theomain}:}}] Fix an arbitrary $\theta>0$, and let
\[
\mathcal{A}(\theta):=\{\mathbf{g}\in H^{*}(\Omega_{\delta},C): \|\mathbf{g}-\Phi\|_{\Delta}\leq \theta\}.
\]
For $\varepsilon>0$ we also define
\[
\mathcal{B}(\varepsilon):=\{\mathbf{g}\in H^{*}(\Omega_{\delta},C): \mathrm{d}(\mathbf{g},\Phi)\leq \varepsilon\}.
\]
There exists $\varepsilon_{0}>0$ such that
\[
\mathcal{B}(\varepsilon_{0})\subset\mathcal{A}(\theta),
\]
otherwise, we could find a sequence of vector functions in $H^{*}(\Omega_{\delta},C)\subset \mathbf{C}_{\Delta}^{+}$ that converges to $\Phi$ in the $\mathrm{d}$-metric but not in the $\|\cdot\|_{\Delta}$-norm which would contradict \eqref{equivconv}.

Let $\Lambda\subset\mathbb{Z}_{+}^{2}(\circledast)$ be an infinite sequence of distinct multi-indices $\mathbf{n}=(n_{1},n_{2})$ such that $\sup_{\mathbf{n}\in\Lambda}(n_{1}-n_{2})<\infty$. Let $0<\varepsilon\leq \varepsilon_{0}$ be fixed, and consider
\[
\mathcal{B}(\varepsilon,\mathbf{n}):=\mathcal{B}(\varepsilon)\cap S^{+}_{\mathbf{n}},\qquad \mathbf{n}\in\Lambda.
\]
We show now that $\mathcal{B}(\varepsilon,\mathbf{n})$ is non-empty for all $\mathbf{n}\in\Lambda$ with norm sufficiently large. Recall that $Q_{\mathbf{n},k}$ is the monic orthogonal polynomial of degree $N_{\mathbf{n},k}$ with respect to the varying measure \eqref{varyingmeasqnk} on $\Delta_{k}$. Let $P_{\mathbf{n},k}$ denote the monic orthogonal polynomial of degree $N_{\mathbf{n},k}$ with respect to the perturbed measure
\[
\frac{\rho_{k}(x)}{\Phi_{k-1}(x)\,\Phi_{k+1}(x)}\,\frac{|h_{\mathbf{n},k}(x)|\,d\sigma_{k}(x)}{|Q_{\mathbf{n},k-1}(x)\,Q_{\mathbf{n},k+1}(x)|}.
\]
Applying Lemmas \ref{relescalar} and \ref{limitehk}, for each $k=1, 2$ we obtain
\begin{equation}\label{convergaux}
\lim_{\mathbf{n}\in\Lambda}\frac{P_{\mathbf{n},k}(z)}{Q_{\mathbf{n},k}(z)}=\frac{\mathsf{G}_{\Delta_{k}}(\rho_{k}/\Phi_{k-1}\Phi_{k+1};z)}{\mathsf{G}_{\Delta_{k}}(\rho_{k}/\Phi_{k-1}\Phi_{k+1};\infty)}=\frac{\Phi_{k}(z)}{\Phi_{k}(\infty)}
\end{equation}
uniformly on compact subsets of $\overline{\mathbb{C}}\setminus\Delta_{k}$. In deriving this it is important to observe that $N_{\mathbf{n},k}$ tends to infinity as $\mathbf{n}$ progresses along the sequence $\Lambda$. It is clear that
\[
\widetilde{\textbf{g}}_{\textbf{n}}:=\left(\frac{\Phi_{1}(\infty) P_{\mathbf{n},1}}{Q_{\mathbf{n},1}},\frac{\Phi_{2}(\infty) P_{\mathbf{n},2}}{Q_{\mathbf{n},2}}\right)\in S_{\mathbf{n}}^{+}\subset \textbf{C}_{\Delta}^{+}
\]
and $\lim_{\mathbf{n}\in\Lambda}\mathrm{d}(\widetilde{\mathbf{g}}_{\mathbf{n}},\Phi)=0$. The convergence \eqref{convergaux} also implies that $\widetilde{\mathbf{g}}_{\mathbf{n}}\in H^{*}(\Omega_{\delta},C)$ for all $\mathbf{n}\in\Lambda$ with norm sufficiently large. Therefore,
there exists $n_{0}>0$ such that $\mathcal{B}(\varepsilon,\mathbf{n})\neq \emptyset$ for $\mathbf{n}\in\Lambda$, $|\mathbf{n}|>n_{0}$.

We prove now that for each $\mathbf{n}\in\Lambda$, $\mathcal{B}(\varepsilon,\textbf{n})$ is compact. Since $\mathcal{B}(\varepsilon,\textbf{n})\subset H^{*}(\Omega_{\delta},C)$ and $H^{*}(\Omega_{\delta},C)$ is compact, it suffices to show that $\mathcal{B}(\varepsilon,\textbf{n})$ is closed. Consider a sequence
\[
(\mathbf{g}_{r})_{r\in\mathbb{N}}=\left(\frac{ P_{r,1}}{Q_{\mathbf{n},1}},\frac{ P_{r,2}}{Q_{\mathbf{n},2}}\right)_{r\in\mathbb{N}}\subset\mathcal{B}(\varepsilon,\mathbf{n})
\]
that converges uniformly on compact sets to $\mathbf{g}\in H^{*}(\Omega_{\delta},C)$. It is clear that $\mathrm{d}(\mathbf{g},\Phi)\leq \varepsilon$, so $\mathbf{g}\in\mathcal{B}(\varepsilon)$. In the finite dimensional space of vector polynomials $(q_{1},q_{2})$ with $\deg(q_{k})\leq N_{\mathbf{n},k}$, we consider the norm
\[
\|(q_{1},q_{2})\|=\left\|\left(\frac{ q_{1}}{Q_{\mathbf{n},1}},\frac{q_{2}}{Q_{\mathbf{n},2}}\right)\right\|_{\Delta}.
\]
With respect to this norm, the sequence $\{(P_{r,1},P_{r,2})\}_{r\in\mathbb{N}}$ is bounded, so there is a subsequence that converges to a vector of polynomials $(P_{1},P_{2})$ and this implies
\[
\mathbf{g}=\left(\frac{ P_{1}}{Q_{\mathbf{n},1}},\frac{ P_{2}}{Q_{\mathbf{n},2}}\right)\in S_{\mathbf{n}}^{+}.
\]
(Note that $\deg P_{k}\leq N_{\mathbf{n},k}$ for each $k$ and $\min_{\Delta}\mathbf{g}\geq C^{-1}$.) This shows that $\mathbf{g}\in\mathcal{B}(\varepsilon,\mathbf{n})$ and so $\mathcal{B}(\varepsilon,\mathbf{n})$ is compact.

Next we show that $\mathcal{B}(\varepsilon,\mathbf{n})$ is convex. Take two arbitrary vector functions
\[
\mathbf{r}_{j}=\left(\frac{P_{j,1}}{Q_{\mathbf{n},1}},\frac{P_{j,2}}{Q_{\mathbf{n},2}}\right)\in\mathcal{B}(\varepsilon,\mathbf{n}),\qquad j=1, 2,
\]
and let $0\leq \alpha\leq 1$. Clearly, for each $k=1, 2$, the polynomial $\alpha P_{1,k}+(1-\alpha) P_{2,k}$ has real coefficients, has degree at most $N_{\mathbf{n},k}$, and $(\alpha P_{1,k}+(1-\alpha) P_{2,k})/Q_{\mathbf{n},k}>0$ on $\Delta_{k-1}\cup\Delta_{k+1}$. Hence $\alpha\,\mathbf{r}_{1}+(1-\alpha)\,\mathbf{r}_{2}\in S_{\mathbf{n}}^{+}$. It is also immediate to check that $\alpha\,\mathbf{r}_{1}+(1-\alpha)\,\mathbf{r}_{2}\in H^{*}(\Omega_{\delta},C)$. So it remains to prove that
\begin{equation}\label{convexityineq}
\left\|\log\left(\frac{\alpha\,P_{1,k}+(1-\alpha)\,P_{2,k}}{Q_{\mathbf{n},k}\,\Phi_{k}}\right)\right\|_{\Delta_{k-1}\cup\Delta_{k+1}}\leq \varepsilon.
\end{equation}
For each $x\in\Delta_{k-1}\cup\Delta_{k+1}$ fixed, $(\alpha P_{1,k}(x)+(1-\alpha) P_{2,k}(x))/Q_{\mathbf{n},k}(x)$ is a value between $P_{1,k}(x)/Q_{\mathbf{n},k}(x)$ and $P_{2,k}(x)/Q_{\mathbf{n},k}(x)$. Hence
\begin{gather*}
0<\min\left\{\frac{  P_{1,k}(x)}{Q_{\mathbf{n},k}(x)\,\Phi_{k}(x)},\frac{ P_{2,k}(x)}{Q_{\mathbf{n},k}(x)\,\Phi_{k}(x)}\right\}\leq \frac{ \alpha P_{1,k}(x)+(1-\alpha) P_{2,k}(x)}{Q_{\mathbf{n},k}(x)\,\Phi_{k}(x)}\\
\leq \max\left\{\frac{  P_{1,k}(x)}{Q_{\mathbf{n},k}(x)\,\Phi_{k}(x)},\frac{ P_{2,k}(x)}{Q_{\mathbf{n},k}(x)\,\Phi_{k}(x)}\right\}.
\end{gather*}
Using the monotonicity of the logarithm, \eqref{convexityineq} readily follows.

Now we show that there exists $\varepsilon' > 0$ such that for all $\mathbf{n} \in \Lambda$ with $|\mathbf{n}|$ sufficiently large, the mapping $\widetilde{T}_{\textbf{n}}$ defined in \eqref{defTntilde} satisfies
\begin{equation}\label{inclusion2}
\widetilde{T}_{\textbf{n}}(\mathcal{B}(\varepsilon',\mathbf{n}))\subset H^{*}(\Omega_{\delta},C).
\end{equation}
Indeed, if this is not true, there exists   a sequence
$(\mathbf{g}_{\bf n})_{{\bf n}\in \Lambda'}, \Lambda' \subset \Lambda$, $\mathbf{g}_{\bf n} \in S_{\mathbf{n}}^{+}\cap H^{*}(\Omega_{\delta},C)$,
such that
\[ \lim_{{\bf n} \in \Lambda'} \mathrm{d}(\mathbf{g}_{\bf n}, \Phi) = 0, \qquad \widetilde{T}_{\bf n}(\mathbf{g}_{\bf n}) \notin H^*(\Omega_\delta, C), \qquad {\bf n} \in \Lambda'.
\]
That is
\begin{equation} \label{inequalities} \| \widetilde{T}_{\bf n}(\mathbf{g}_{\bf n})\|_{\Omega_\delta}  > C, \quad \mbox{or} \quad \min_{\Delta}  \widetilde{T}_{\bf n}(\mathbf{g}_{\bf n}) < C^{-1}, \qquad {\bf n} \in \Lambda'.
\end{equation}
 The zeros of the $k$-th component of ${\bf g}_{\bf n}$ remain uniformly bounded away from $\Delta_{k-1} \cup \Delta_{k+1}$ for ${\bf n} \in \Lambda'$ since the $k$-th component $\Phi_k$ of $\Phi$ has no zero on that set. Therefore, we can apply Lemma \ref{relvectorial} and from \eqref{limfund3} and \eqref{equivconv} it follows that
\[ \lim_{{\bf n} \in \Lambda'} \widetilde{T}_{\bf n}({\bf g}_{\bf n}) = T(\Phi) = \Phi,
\]
componentwise and uniformly on each compact subset of $\Omega=(\Omega_{1},\Omega_{2})$, where $\Omega_{k}=\overline{\mathbb{C}}\setminus\Delta_{k}$. This implies that
\[ \lim_{{\bf n} \in \Lambda'}
\|\widetilde{T}_{\bf n}({\bf g}_{\bf n})\|_{\Omega_\delta} = \|\Phi\|_{\Omega_{\delta}} < C \qquad \mbox{and} \qquad
 \lim_{{\bf n} \in \Lambda'}
\min_\Delta \widetilde{T}_{\bf n}({\bf g}_{\bf n})  = \min_{\Delta} \Phi > C^{-1}
\]
which contradicts both inequalities in \eqref{inequalities}. Therefore, \eqref{inclusion2} holds. In what follows, we assume that the constant $\varepsilon$ that was previously fixed is chosen so that
\begin{equation}\label{newinclusion}
\widetilde{T}_{\textbf{n}}(\mathcal{B}(\varepsilon,\mathbf{n}))\subset H^{*}(\Omega_{\delta},C)
\end{equation}
takes place for all $\mathbf{n}\in\Lambda$ with $|\mathbf{n}|$ sufficiently large.

We show now that
\begin{equation}\label{inclusion}
\widetilde{T}_{\textbf{n}}(\mathcal{B}(\varepsilon,\mathbf{n}))\subset\mathcal{B}(\varepsilon,\mathbf{n})
\end{equation}
for $|\mathbf{n}|$ large enough. If we assume the contrary, since \eqref{newinclusion} takes place, there exists a sequence
$(\mathbf{g}_{\bf n})_{{\bf n}\in \Lambda'}, \Lambda' \subset \Lambda$, $\mathbf{g}_{\bf n} \in \mathcal{B}(\varepsilon,{\bf n})$
such that
\begin{equation} \label{paracontra}
\mathrm{d}(\widetilde{T}_{\bf n}(\mathbf{g}_{\bf n}  ),\Phi ) = \mathrm{d}(\widetilde{T}_{\bf n}(\mathbf{g}_{\bf n}  ),T(\Phi)) > \varepsilon\,, \qquad {\bf n} \in \Lambda'\,.
\end{equation}
However, the sequence $(\mathbf{g}_{\bf n})_{{\bf n}\in \Lambda'}$ is uniformly bounded on each compact subset of $\Omega_{\delta}$; therefore, there exists a subsequence of multi-indices $\Lambda'' \subset \Lambda'$ and $\mathbf{g} \in
\mathcal{B}(\varepsilon)$ such that $(\mathbf{g}_{\bf n})_{{\bf n}\in \Lambda''}$ converges to ${\bf g}$ uniformly on each compact subset of $\Omega_{\delta}$.  In particular,  $\lim_{{\bf n} \in \Lambda''} \mathrm{d}(\mathbf{g}_{\mathbf{n}}, \mathbf{g}) = 0$. Notice that for each $k=1, 2$, the zeros of the $k$-th component of ${\bf g}_{\bf n}$ remain uniformly bounded away from $\Delta_{k-1} \cup \Delta_{k+1}$ for ${\bf n} \in \Lambda''$ since the $k$-th component of $\bf g$ has no zero on this set (observe that $\min_{\Delta}\mathbf{g}_{\mathbf{n}}\geq C^{-1}$ for all $\mathbf{n}\in\Lambda'$). Therefore, we can apply Lemma \ref{relvectorial} and from \eqref{limfund3} and \eqref{equivconv} it follows that
\begin{equation}\label{limpotencia}
\lim_{{\bf n} \in \Lambda''} \mathrm{d}(\widetilde{T}_{\bf n}({\bf g}_{\bf n}),T({\bf g})) = 0.
\end{equation}
The triangular inequality and \eqref{contractionprop} imply
\[ \mathrm{d}(\widetilde{T}_{\bf n}({\bf g}_{\bf n}), \Phi) \leq   \mathrm{d}(\widetilde{T}_{\bf n}({\bf g}_{\bf n}), T({\bf g})) + \mathrm{d}({T}({\bf g}), T(\Phi)) \leq  \mathrm{d}(\widetilde{T}_{\bf n}({\bf g}_{\bf n}), T({\bf g})) +
\frac{\varepsilon}{2}, \quad {\bf n} \in \Lambda.
\]
Using \eqref{limpotencia}, we can make the first term after the last inequality $<\varepsilon/2$ for all ${\bf n} \in \Lambda''$ with $|{\bf n}|$ sufficiently large, say $|{\bf n}| > N$. Therefore,
\[ \mathrm{d}(\widetilde{T}_{\bf n}({\bf g}_{\bf n}), \Phi) < \varepsilon, \qquad {\bf n} \in \Lambda'', \quad
|{\bf n}| > N.\]
But then, for such values of ${\bf n} \in \Lambda''$  we get a contradiction with \eqref{paracontra}. Consequently, \eqref{inclusion} takes place
for all $ {\bf n} \in \Lambda$ with $|{\bf n}|$ sufficiently large.

Summarizing, we have proved that for all ${\bf n} \in \Lambda$ with $|{\bf n}| \geq n_0$, $n_0$ sufficiently large, $\mathcal{B}(\varepsilon,{\bf n})$ is a nonempty, convex, compact set (in the topology of uniform convergence on compact sets in $H(\Omega_{\delta})$) such that \eqref{inclusion} holds. Since $\widetilde{T}_{\mathbf{n}}$ is continuous, by the generalized Brouwer fixed point theorem (see \cite[Theorem V.19]{ReedSimon}), $\widetilde{T}_{\bf n}$ has a fixed point in $\mathcal{B}(\varepsilon, {\bf n})$ for all $\mathbf{n}\in\Lambda$ with $|\mathbf{n}|$ large enough. But we know by Theorem~\ref{fixedpointTntilde} that $\widetilde{T}_{\mathbf{n}}$ has a unique fixed point in the entire space $S_{\mathbf{n}}^{+}$, which is given by the vector
\begin{equation}\label{def:gnstar}
\mathbf{g}_{\mathbf{n}}^{*}=\left(c_{\mathbf{n},1}\frac{\widetilde{Q}_{\mathbf{n},1}}{Q_{\mathbf{n},1}},c_{\mathbf{n},2}\frac{\widetilde{Q}_{\mathbf{n},2}}{Q_{\mathbf{n},2}}\right),
\end{equation}
where
$
c_{\mathbf{n},1}=\left(\widetilde{\kappa}_{\mathbf{n},1}/\kappa_{\mathbf{n},1}\right)^{4/3}\left(\widetilde{\kappa}_{\mathbf{n},2}/\kappa_{\mathbf{n},2}\right)^{2/3}$ and  $c_{\mathbf{n},2}=\left(\widetilde{\kappa}_{\mathbf{n},1}/\kappa_{\mathbf{n},1}\right)^{2/3}\left(\widetilde{\kappa}_{\mathbf{n},2}/\kappa_{\mathbf{n},2}\right)^{4/3}
$. Hence $\mathbf{g}_{\mathbf{n}}^{*}\in\mathcal{B}(\varepsilon,\mathbf{n})$ for all $\mathbf{n}\in\Lambda$ with $|\mathbf{n}|$ large enough.

We have shown that for each $\theta > 0$ there exists $\varepsilon > 0$ and $n_1 \in \mathbb{N}$
 such that for all ${\bf n} \in \Lambda$, $|{\bf n}| > n_1$, the vector defined in \eqref{def:gnstar} satisfies $\mathbf{g}_{\mathbf{n}}^{*}\in\mathcal{B}(\varepsilon,\mathbf{n})\subset\mathcal{A}(\theta)$. It follows that
\[ \lim_{{\bf n} \in \Lambda} \| \mathbf{g}_{\bf n}^*    - \Phi \|_{\Delta} = 0.
\]
Then, since $\mathbf{g}_{\mathbf{n}}^{*}$ is a fixed point of $\widetilde{T}_{\mathbf{n}}$ and $\Phi$ is a fixed point of $T$, applying Lemma \ref{relvectorial} we obtain
\begin{equation}\label{limitfinal}
\lim_{\mathbf{n} \in \Lambda}  c_{\mathbf{n},k}\,\frac{\widetilde{Q}_{{\bf n},k}(z)}{Q_{{\bf n},k}(z)} = \Phi_k(z), \qquad k=1, 2,
\end{equation}
uniformly on compact subsets of $\overline{\mathbb{C}} \setminus \Delta_k$. In particular we have
\begin{equation}\label{asympcnk}
\lim_{\mathbf{n} \in \Lambda}c_{\mathbf{n},k}=\Phi_{k}(\infty),\qquad k=1, 2,
\end{equation}
and therefore the limit \eqref{limfund1} holds.

Since $\widetilde{\kappa}_{\mathbf{n},1}/\kappa_{\mathbf{n},1}=c_{\mathbf{n},1}\,c_{\mathbf{n},2}^{-1/2}$ and $\widetilde{\kappa}_{\mathbf{n},2}/\kappa_{\mathbf{n},2}=c_{\mathbf{n},2}\,c_{\mathbf{n},1}^{-1/2}$, from \eqref{limitfinal}, \eqref{asympcnk}, and \eqref{def:littleqnk} we also deduce that
\begin{equation}\label{limratiokappas}
\lim_{\mathbf{n}\in\Lambda}\frac{\widetilde{\kappa}_{\mathbf{n},1}}{\kappa_{\mathbf{n},1}}=\frac{\Phi_{1}(\infty)}{\Phi_{2}(\infty)^{1/2}},\qquad \lim_{\mathbf{n}\in\Lambda}\frac{\widetilde{\kappa}_{\mathbf{n},2}}{\kappa_{\mathbf{n},2}}=\frac{\Phi_{2}(\infty)}{\Phi_{1}(\infty)^{1/2}},
\end{equation}
and
\[
\lim_{\mathbf{n}\in\Lambda}\frac{\widetilde{q}_{\mathbf{n},1}(z)}{q_{\mathbf{n},1}(z)}=\frac{\Phi_{1}(z)}{\Phi_{2}(\infty)^{1/2}},\qquad \lim_{\mathbf{n}\in\Lambda}\frac{\widetilde{q}_{\mathbf{n},2}(z)}{q_{\mathbf{n},2}(z)}=\frac{\Phi_{2}(z)}{\Phi_{1}(\infty)^{1/2}}.
\]
From \eqref{defKnk} and \eqref{def:kappas} it follows that
\[
\frac{\widetilde{K}_{\mathbf{n},k}}{K_{\mathbf{n},k}}=\prod_{j=1}^{k}\frac{\widetilde{\kappa}_{\mathbf{n},j}}{\kappa_{\mathbf{n},j}},\qquad k=1, 2,
\]
and, therefore, from \eqref{limratiokappas} we obtain
\begin{equation}\label{limrelKnks}
\lim_{\mathbf{n}\in\Lambda}\frac{\widetilde{K}_{\mathbf{n},1}}{K_{\mathbf{n},1}}=\frac{\Phi_{1}(\infty)}{\Phi_{2}(\infty)^{1/2}},\qquad \lim_{\mathbf{n}\in\Lambda}\frac{\widetilde{K}_{\mathbf{n},2}}{K_{\mathbf{n},2}}=(\Phi_{1}(\infty)\,\Phi_{2}(\infty))^{1/2}.
\end{equation}
A combination of \eqref{Hnk}, \eqref{def:littlehnk}, and simple algebra gives the identity
\[
\frac{\widetilde{\Psi}_{\mathbf{n},k}(z)}{\Psi_{\mathbf{n},k}(z)}=
\frac{\widetilde{Q}_{\mathbf{n},k+1}(z)}{Q_{\mathbf{n},k+1}(z)}\,
\frac{Q_{\mathbf{n},k}(z)}{\widetilde{Q}_{\mathbf{n},k}(z)}\,
\frac{K_{\mathbf{n},k}^{2}}{\widetilde{K}_{\mathbf{n},k}^{2}}\,
\frac{\widetilde{\varepsilon}_{\mathbf{n},k}\,
\widetilde{h}_{\mathbf{n},k+1}(z)}{\varepsilon_{\mathbf{n},k}\,h_{\mathbf{n},k+1}(z)}\,
\frac{\widetilde{\varepsilon}_{\mathbf{n},k}}{\varepsilon_{\mathbf{n},k}},\qquad k=1, 2,
\]
where $\varepsilon_{\mathbf{n},k}$ and $\widetilde{\varepsilon}_{\mathbf{n},k}$ are the signs of the measures \eqref{varyingmeasqnk:2} and
\begin{equation}\label{varyingmeasqnk:3}
{\widetilde{h}_{{\bf n},k}(x)\,\rho_{k}(x)\,d\sigma_k(x)}/{\widetilde{Q}_{{\bf n},k-1}(x)\,\widetilde{Q}_{{\bf n},k+1}(x)},
\end{equation}
respectively. Recall that $\widetilde{Q}_{\mathbf{n},3}\equiv Q_{\mathbf{n},3}\equiv 1$.

The measures $\sigma_{k}$ and $\rho_{k}\,\sigma_{k}$ are positive, and the polynomials $\widetilde{Q}_{\mathbf{n},k}$ and $Q_{\mathbf{n},k}$ have the same degree and their zeros are located on the same interval $\Delta_{k}$. As a consequence of these facts and the integral representations \eqref{entreHnks}, it is easy to see that the measures \eqref{varyingmeasqnk:2} and \eqref{varyingmeasqnk:3} have the same sign, i.e. $\varepsilon_{\mathbf{n},k}=\widetilde{\varepsilon}_{\mathbf{n},k}$ for each $k=1, 2$ and $\mathbf{n}\in\Lambda$. Then, applying \eqref{limfund1}, \eqref{limrelKnks}, \eqref{limenkhnk}, and \eqref{convhkfunc}, we finally obtain the limits
\begin{align*}
\lim_{\mathbf{n}\in\Lambda}\frac{\widetilde{\Psi}_{\mathbf{n},1}(z)}{\Psi_{\mathbf{n},1}(z)} & =\frac{1}{\Phi_{1}(\infty)}\frac{\Phi_{2}(z)}{\Phi_{1}(z)},\qquad z\in\overline{\mathbb{C}}\setminus(\Delta_{1}\cup\Delta_{2}),\\
\lim_{\mathbf{n}\in\Lambda}\frac{\widetilde{\Psi}_{\mathbf{n},2}(z)}{\Psi_{\mathbf{n},2}(z)}& =\frac{1}{\Phi_{1}(\infty)\,\Phi_{2}(z)},\qquad z\in\overline{\mathbb{C}}\setminus\Delta_{2},
\end{align*}
Therefore, \eqref{limfund2} and \eqref{limfund4} are justified and we are done. \end{proof}

Regarding the question posed in the introduction relative to the generalization of Theorem \ref{theomain} to the case of $m>2$ generating measures, the situation is the following. The scheme of the proof seems to work well. The operator
$T$ ceases to be contractive but $T^{\overline{m}}$, $\overline{m} = \lceil m/2 \rceil$, is contractive (see \cite[Theorem 1.6]{luisyo2}) and Brouwer's theorem can be applied to prove that $\widetilde{T}_{\mathbf{n}}^{\overline{m}}$ has fixed points in all sufficiently small neighborhoods of the fixed point of $T$. Unfortunately, such fixed points may not be unique and we have not been
able to find their link with  the analogue of \eqref{def:gnstar} in the general case.

\end{document}